\documentclass[11pt,leqno]{article}
\usepackage{amsmath, amscd, amsthm, amssymb, graphics, xypic, mathrsfs, setspace, fancyhdr, times, bm, enumitem}
\usepackage[colorlinks=true,pagebackref=true]{hyperref} 
\hypersetup{backref}

\newcommand{\Spec}{\operatorname{Spec}}

\newcommand{\isomto}{{\stackrel{\sim}{\;\longrightarrow\;}}}
\newcommand{\isomt}{{\stackrel{{\scriptscriptstyle{\sim}}}{\;\rightarrow\;}}}

\newcommand{\sma}{{\scriptstyle{\wedge}}}

\renewcommand{\O}{{\mathcal O}}
\renewcommand{\hom}{\operatorname{Hom}}

\newcommand{\real}{{\mathbb R}}
\newcommand{\cplx}{{\mathbb C}}

\newcommand{\Z}{{\mathbb Z}}

\newcommand{\aone}{{\mathbb A}^1}
\newcommand{\pone}{{\mathbb P}^1}

\newcommand{\ga}{{{\mathbb G}_{a}}}
\newcommand{\gm}[1]{{{\mathbf G}_{\mathrm m}^{#1}}}

\newcommand{\ho}[1]{\mathscr{H}({#1})}
\newcommand{\hop}[1]{\mathscr{H}_{\bullet}({#1})}

\newcommand{\bpi}{\bm{\pi}}

\newcommand{\Nis}{\operatorname{Nis}}
\newcommand{\Zar}{\operatorname{Zar}} 

\newcommand{\Sm}{\mathrm{Sm}}

\newcommand{\Spc}{\mathrm{Spc}}

\newcommand{\K}{{{\mathbf K}}}

\newcommand{\hsnis}{\mathscr{H}_s^{\Nis}(k)}
\newcommand{\hspnis}{\mathscr{H}_{s,\bullet}^{\Nis}(k)}

\newcommand{\F}{{\mathcal F}}

\newcommand{\Sing}{\mathrm{Sing}}

\newcommand{\Addresses}{{
  \bigskip
  \footnotesize

  A.~Asok, \textsc{Department of Mathematics, University of Southern California, 3620 S. Vermont Ave. KAP 104,
    Los Angeles, CA 90089-2532, United States;} \textit{E-mail address:} \url{asok@usc.edu}

  \medskip

  B.~Doran, \textsc{Departement Mathematik, Eidgen\"ossische Technische Hochschule, R{\"a}mistrasse 101, 8092 Z\"urich Switzerland;} \textit{E-mail address:} \url{brent.doran@math.ethz.ch}

  \medskip

  J.~Fasel, \textsc{Institut Fourier - UMR 5582, Universit\'e Grenoble Alpes, 100 rue des math\'ematiques BP 74, F-38000 Grenoble;} \textit{E-mail address:} \url{jean.fasel@gmail.com}

}}

\newcounter{intro}
\theoremstyle{plain}
\newtheorem{thm}{Theorem}[subsection]

\newtheorem{lem}[thm]{Lemma}
\newtheorem{cor}[thm]{Corollary}
\newtheorem{prop}[thm]{Proposition}
\newtheorem*{claim*}{Claim}  

\newtheorem{question}[thm]{Question}

\newtheorem{conj}[thm]{Conjecture}
\newtheorem*{thm*}{Theorem}
\newtheorem*{problem*}{Problem}

\newtheorem{thmintro}{Theorem}
\newtheorem{propintro}[thmintro]{Proposition}

\newtheorem{questionintro}[thmintro]{Question}

\theoremstyle{definition}

\theoremstyle{remark}
\newtheorem{rem}[thm]{Remark}
\newtheorem{remintro}[thmintro]{Remark}

\newtheorem{ex}[thm]{Example}

\numberwithin{equation}{section}

\begin{document}
\pagestyle{fancy}
\renewcommand{\sectionmark}[1]{\markright{\thesection\ #1}}
\fancyhead{}
\fancyhead[LO,R]{\bfseries\footnotesize\thepage}
\fancyhead[LE]{\bfseries\footnotesize\rightmark}
\fancyhead[RO]{\bfseries\footnotesize\rightmark}
\chead[]{}
\cfoot[]{}
\setlength{\headheight}{1cm}

\author{Aravind Asok\thanks{Aravind Asok was partially supported by National Science Foundation Award DMS-1254892.} \and Brent Doran\thanks{Brent Doran was partially supported by Swiss National Science Foundation Award 200021\_138071.} \and Jean Fasel}

\title{{\bf Smooth models of motivic spheres\\ and the clutching construction}}
\date{}
\maketitle

\begin{abstract}
We study the representability of motivic spheres by smooth varieties.  We show that certain explicit ``split" quadric hypersurfaces have the $\aone$-homotopy type of motivic spheres over the integers and that the $\aone$-homotopy types of other motivic spheres do not contain smooth schemes as representatives.  We then study some applications of these representability/non-representability results to the construction of new exotic $\aone$-contractible smooth schemes.  Then, we study vector bundles on even dimensional ``split" quadric hypersurfaces by developing an algebro-geometric variant of the classical construction of vector bundles on spheres via clutching functions.
\end{abstract}

\begin{footnotesize}
\setcounter{tocdepth}{1}
\tableofcontents
\end{footnotesize}

\section{Introduction}
This note is concerned with several results about spheres in the Morel-Voevodsky $\aone$-homotopy theory \cite{MV}, also known as motivic spheres.  First, we make precise the idea that (split) smooth affine quadric hypersurfaces are motivic spheres.  Second, we try to better understand the situations in which motivic spheres admit models as smooth schemes.  Finally, we use homotopic ideas in an attempt to understand how to explicitly construct all vector bundles on split smooth affine quadrics.  We now explain our results more precisely.

Given a commutative ring $k$, the Morel-Voevodsky $\aone$-homotopy category $\ho{k}$ is constructed by first enlarging the category $\Sm_k$ of schemes smooth over $\Spec k$ to a category $\Spc_k$ of spaces over $\Spec k$ and then performing a categorical localization.  The category $\Spc_k$ has all small limits and colimits, but the functor $\Sm_k \to \Spc_k$ does not preserve all colimits that exist in $\Sm_k$, e.g., quotients in $\Sm_k$ that exist need not coincide with quotients computed in $\Spc_k$.

The resulting homotopy theory of schemes bears a number of similarities to the classical homotopy category.  For example, there are two objects of $\Spc_k$ that play a role analogous to that played by the circle in classical homotopy theory.  On the one hand, one can view the circle as $I/\{0,1\}$ where $I$ is the unit interval $[0,1]$.  If we view ${\mathbb A}^1$ as analogous to the unit interval, we can form the quotient space ${\mathbb A}^1/\{0,1\}$ (pointed by the image of $\{0,1\}$ in the quotient), and we call this space $S^1_s$.  This quotient does exist as a singular scheme (i.e., a nodal curve).  However, it is not {\em a priori} clear whether $S^1_s$ is isomorphic in $\ho{k}$ to a smooth scheme.  On the other hand, if we emphasize the fact that $S^1$ has a group structure, then we could also think of $\gm{}$ (pointed by $1$) as a version of the circle.  Of course, $\gm{}$ is a smooth $k$-scheme.

General motivic spheres are obtained by taking smash products (this notion makes sense in $\Spc_k$ because of existence of colimits) of copies of $S^1_s$ and $\gm{}$ \cite[\S 3.2]{MV}; by convention, we set $S^0_s$ to be the disjoint union of two copies of $\Spec k$.  While in classical topology spheres are by construction smooth manifolds, the example of $S^1_s$ shows that, in stark contrast, the following question in $\aone$-homotopy theory has no obvious answer.

\begin{questionintro}
\label{questionintro:main}
Which motivic spheres $S^i_s \wedge \gm{\wedge j}$ have the $\aone$-homotopy type of a smooth scheme?
\end{questionintro}

Throughout this paper, we consider the smooth affine quadric hypersurfaces
\begin{equation*}
\begin{split}
&Q_{2m-1} := \Spec k[x_1,\ldots,x_m,y_1,\ldots,y_m]/\langle \sum_i x_iy_i - 1 \rangle, \text{ and } \\
&Q_{2m} := \Spec k[x_1,\ldots,x_m,y_1,\ldots,y_m,z]/\langle \sum_i x_i y_i - z(1 + z)\rangle.
\end{split}
\end{equation*}
One immediately verifies that all of these hypersurfaces are in fact smooth over $\Spec k$.

Note that $Q_0$ is the subscheme of ${\mathbb A}^1$ given by the two points $0,1$ and thus coincides with $S^0_s$.  The hypersurface $Q_{2m-1}$ is well-known to be $\aone$-weakly equivalent to ${\mathbb A}^m \setminus 0$ (by projection onto $x_1,\ldots,x_m$) and the latter is $\aone$-weakly equivalent to the motivic sphere $S^{m-1}_s \wedge \gm{\sma m}$ \cite[Example 2.20]{MV}. Our first result provides an explicit description of the $\aone$-homotopy type of the other family of quadrics, thus answering Question \ref{questionintro:main} in a collection of cases.

\begin{thmintro}[See Theorem \ref{thm:main}]
\label{thmintro:main}
Let $k$ be a commutative unital ring and $n \geq 0$ an integer.  There are explicit $\aone$-weak equivalences
\[
Q_n \sim_{\aone} \begin{cases}
S^{m-1}_s \wedge \gm{\wedge m} & \text{ if } n = 2m-1 \text{ and}\\
S^m_s \wedge \gm{\wedge m} & \text{ if } n  = 2m
\end{cases}
\]
of spaces over $\Spec k$.
\end{thmintro}

\begin{remintro}
The interesting part of Theorem \ref{thmintro:main} is the case where $n = 2m$, and the result was known previously for $m = 1,2$ by special geometric considerations.  On the other hand, the $S^1$-stable homotopy types of the quadrics $Q_{2m}$ were known over fields having characteristic unequal to $2$ by unpublished work of Morel \cite{Morelquadrics}, and Dugger and Isaksen \cite{DuggerIsaksenHopf}.  We emphasize that Theorem \ref{thmintro:main} is an {\em unstable} result.  In contrast to the corresponding proof for $n$ odd, it uses the Nisnevich topology in an essential way, via the Morel-Voevodsky homotopy purity theorem.
\end{remintro}

While Theorem \ref{thmintro:main} tells us that some motivic spheres {\em do} have the $\aone$-homotopy type of smooth schemes, the following result shows that not all motivic spheres have this property.

\begin{propintro}[See Proposition \ref{prop:nonrepresentability}]
\label{propintro:nonrepresentability}
If $i,j$ are integers with $i > j$, the spheres $S^i_s \wedge \gm{\wedge j}$ do not have the $\aone$-homotopy type of smooth (affine) schemes.
\end{propintro}

Classically, $S^n$ admits an open cover by two contractible subspaces homemorphic to ${\mathbb R}^n$ (via stereographic projection), namely $S^n \setminus \{N\}$ and $S^n \setminus \{S\}$, where $N$ and $S$ are the North and South poles.  The intersection of these two open sets is homeomorphic to the total space of a line bundle over $S^{n-1}$ and, in particular, homotopy equivalent to $S^{n-1}$.  Pursuing the analogy between the classical spheres and the quadrics further, one can ask whether analogous covers exist in algebraic geometry.  We begin by establishing the existence of $\aone$-contractible open subschemes covering $Q_{2n}$; the main results are summarized in the following theorem (see Subsection \ref{ss:q2andq4} and Remark \ref{rem:alternativeproof} for detailed explanations of the connections with \cite{ADContractible}).

\begin{thmintro}[See Theorem \ref{thm:newaonecontractibles} and Corollary \ref{cor:notaunipotentquotient}]
\label{thmintro:contractible}
Suppose $k$ is a commutative unital ring.  Write $X_{2m}$ for the open subscheme of $Q_{2m}$ defined as the complement of $x_1 = \cdots = x_m = 0, z = -1$.
\begin{enumerate}[noitemsep,topsep=1pt]
\item The variety $X_{2m}$ is $\aone$-contractible over $\Spec k$.
\item For $m \geq 3$, $X_{2m}$ cannot be realized as a quotient of an affine space by a (scheme-theoretically) free action of a unipotent $k$-group scheme.
\end{enumerate}
\end{thmintro}

The identification of the $\aone$-homotopy type of spheres has a number of applications.  Pursing the analogy with the classical geometry of $S^n$ further, recall that the open cover of $S^n$ described above also yields the ``clutching construction" which provides a bijection between pointed homotopy classes of maps $S^{n-1} \to GL_n(\real)$ and isomorphism classes of real vector bundles on $S^{n}$.  We pursue an analogous construction in algebraic geometry.  The open subscheme $X_{2m}$ has an ``opposite" defined as the complement of $x_1 = \cdots = x_m = z = 0$; this subscheme is $\aone$-contractible as well (if $2$ is a unit in $k$, then $X_{2m}$ and its opposite are isomorphic algebraic varieties).  When $m = 1$, the intersection of $X_{2m}$ and its opposite is isomorphic to $\aone \times \gm{}$.  For $m \geq 2$, the situation is more complicated as the intersection of $X_{2m}$ and its opposite is a quasi-affine but not affine variety, so every regular function on the intersection extends to $Q_{2m}$ itself.  As a consequence, considering the intersection of $X_{2m}$ and its opposite does not produce a reasonable analog of the clutching construction.

Instead, we produce two variants (abstract and concrete) of the clutching construction (see Theorems \ref{thm:abstractclutching} and \ref{thm:concreteclutching}).  The abstract clutching construction works rather generally and, roughly speaking, yields a surjection from the set of ``naive" $\aone$-homotopy classes of pointed maps from $Q_{2m-1}$ to $GL_r$ to the set of isomorphism classes of rank $r$ vector bundles on $Q_{2m}$; this result holds over $\Z$ (or, more generally, any base ring that is smooth over a Dedekind domain with perfect residue fields).

The concrete clutching construction refines the abstract clutching construction and shows how to produce explicit $GL_m$-valued $1$-cocycles on a suitable open cover of $Q_{2m}$ realizing any rank $r$ vector bundle, at least under suitable additional restrictions on the base $k$.  More precisely, consider the open affine subschemes $V_{2m}^0 := D_z \subset Q_{2m}$ and $V_{2m}^1 := D_{1+z} \subset Q_{2m}$.  The subschemes $V_{2m}^1$ and $V_{2m}^0$ are {\em not} $\aone$-contractible.  Nevertheless, there is an explicit morphism $\psi_m: V_{2m}^0 \cap V_{2m}^1 \to Q_{2m-1}$ such that the vector bundle on $Q_{2m}$ attached to a pointed morphism $f: Q_{2m-1} \to GL_r$ is obtained by gluing copies of the trivial bundle of rank $r$ on $V_{2m}^i$ along $V_{2m}^0 \cap V_{2m}^1$ via the composite map $f \circ \psi_m$.  One consequence of this construction is the following result, which is perhaps of independent interest.

\begin{thmintro}[See Theorem \ref{thm:concreteclutching}(2)]
\label{thmintro:generator}
If $k$ is a regular ring, then for every integer $n \geq 1$, there exists an explicit matrix $\beta_m: Q_{2m-1} \to GL_m$, constructed by Suslin, such that the rank $n$ vector bundle on $Q_{2m}$ obtained by gluing trivial bundles via the automorphism $\beta_m \circ \psi_m$ yields a generator of the rank $1$ free $K_0(k)$-module $\tilde{K}_0(Q_{2m})$.
\end{thmintro}

These results have a number of applications which will be developed elsewhere, but which we mention here for the sake of context.  Naive homotopy classes of maps from a smooth affine $k$-scheme $X$ to $Q_{2n}$ appear in problems related to complete intersection ideals \cite[\S 2]{FaselMurthy}.  On the other hand, the set of naive $\aone$-homotopy classes of maps from a smooth affine scheme to $Q_{2n}$ was identified, at least if $k$ is an infinite field having characteristic unequal to $2$, with $[X,Q_{2n}]_{\aone}$ in \cite[Theorem 4.2.2]{AHWII}.  These ideas will be put together with the results established here in \cite{AsokFaselmotiviccohomotopy}.  Among other things, the set $[X,Q_{2m}]_{\aone}$ is a ``motivic cohomotopy group" (provided the dimension of $X$ is "small" compared to $m$): it can be equipped with a functorial abelian group structure and this additional structure can be used to study various classical problems, e.g., related to comparison of ``Euler class groups."

\subsubsection*{Note on prerequisites}
We close this introduction with a warning: the prerequisites required to read this paper are non-uniform.  The proofs of Theorem \ref{thmintro:main} and (the first part of) Theorem \ref{thmintro:contractible} use only basic algebraic geometry, standard results about cofiber sequences in pointed model categories \cite[Chapter 6]{Hovey} and knowledge of the basic aspects of \cite{MV} (see the preliminaries for further discussion regarding relaxation of hypotheses on the base); we have attempted to make them as accessible as possible.  The proofs of Proposition \ref{propintro:nonrepresentability} and the second part of Theorem \ref{thmintro:contractible} are slightly more algebro-geometrically involved: the first requires some knowledge of properties of motivic cohomology \cite{MVW} while the second requires knowledge about Cousin complexes for coherent sheaves.  The results of Section \ref{s:vectorbundles} on the clutching construction require some familiarity with some of the results of \cite{AHW,AHWII} (which we review), together with some ideas from \cite{MField}.  The proof of Theorem \ref{thmintro:generator} requires familiarity with some results of Suslin.

\subsubsection*{Acknowledgements}
Some of the results claimed in this paper were announced in talks by first and second named authors as early as 2007.  For example, the original proof of Theorem \ref{thmintro:main} we envisioned, specifically the description of $\aone$-homotopy types of $Q_n$ for $n$ even $\geq 6$, was an elementary geometric argument that was supposed to work over $\Spec \Z$; this argument contained a gap.  The authors would also like to thank Paul Balmer, Dan Isaksen, Christian Haesemeyer, and Fabien Morel, for useful discussions over the course of the project and Matthias Wendt for a number of useful comments on a first draft of this note.  Finally, we thank Marc Hoyois for helpful discussions about the $\aone$-homotopy category over non-Noetherian base schemes.

\subsubsection*{Preliminaries/Notation}
Throughout the paper the phrase ``$k$ is a ring" will always mean ``$k$ is a commutative unital ring".  We write $\Sm_k$ for the category of schemes that are separated, finite type and smooth over $\Spec k$, and $\Spc_k$, i.e., the category of ``spaces over $k$", is the category of simplicial presheaves on $\Sm_k$.  We write $\hsnis$ (resp. $\hspnis$) for the Nisnevich local homotopy category (as in \cite[\S 2.1]{MV}) and $\ho{k}$ (resp. $\hop{k}$) for the (pointed) Morel-Voevodsky $\aone$-homotopy category \cite{MV}.

The assumptions here require further discussion if $k$ is not Noetherian of finite Krull dimension, which are the explicit assumptions on $k$ used in $\cite{MV}$ in the construction of the Morel-Voevodsky $\aone$-homotopy category.  In contrast to \cite{MV}, but following \cite{AHW}, we define the Nisnevich topology on $\Sm_k$ to be the topology generated by finite families of \'etale maps $\{U_i \to X\}$ admitting a splitting sequence by finitely presented closed subschemes, in the sense of \cite[\S3, p. 97]{MV}.  This definition, which is equivalent to the ``standard" definition \cite[\S 3 Definition 1.2]{MV} if $k$ is Noetherian by \cite[\S3 Lemma 1.5]{MV}, is studied in more detail in \cite[\S1]{DAGXI}.   We refer the reader to \cite[Appendix C]{Hoyois} for a ``good" version of the $\aone$-homotopy category over a quasi-compact, quasi-separated base scheme.  We note here that arbitrary affine schemes are quasi-compact and quasi-separated. In addition, in the construction of the $\aone$-homotopy category when $k$ is not Noetherian, one imposes Nisnevich descent (not hyperdescent) as discussed in \cite[\S 3.1]{AHW} and $\aone$-invariance.

As usual, smooth schemes are, via the Yoneda embedding, identified with their corresponding representable (simplicially constant) presheaves.  If $\mathcal{X} \in \Spc_k$, we write $\mathcal{X}_+$ for the pointed space $\mathcal{X} \coprod \Spec k$, pointed by the disjoint copy of $\Spec k$.  If $Y \to X$ is a closed immersion of smooth scheme, we write $\nu_{Y/X}$ for the normal bundle to $Y$ in $X$.  In that case, $Th(\nu_{Y/X})$ is the Thom space of the vector bundle $\nu_{Y/X}$ as in \cite[\S 3 Definition 2.16]{MV}.

Suppose $(\mathscr{X},x)$ and $(\mathscr{Y},y)$ are pointed spaces.  We use the notation $\Sigma^i_s \mathscr{X} := S^i_s \wedge \mathscr{X}$ for simplicial suspension, and suppress the base-point.  Likewise, we write ${\mathbf R}\Omega^1_s \mathscr{X}$ for the derived simplicial loops of $\mathscr{X}$: i.e., first apply Nisnevich fibrant replacement and then apply $\Omega^1_s$.  We set $[(\mathscr{X},x),(\mathscr{Y},y)]_s := \hom_{\hspnis}((\mathscr{X},x),(\mathscr{Y},y))$, $[(\mathscr{X},x),(\mathscr{Y},y)]_{\aone} := \hom_{\hop{k}}((\mathscr{X},x),(\mathscr{Y},y))$ $[\mathscr{X},\mathscr{Y}]_{s} := \hom_{\hsnis}(\mathscr{X},\mathscr{Y})$ and $[\mathscr{X},\mathscr{Y}]_{\aone} := \hom_{\ho{k}}(\mathscr{X},\mathscr{Y})$.  Finally, if $(\mathcal{X},x)$ is a pointed space, we write $\bpi_{i}^{\aone}(\mathcal{X},x)$ for the Nisnevich sheaf associated with the presheaf $U \mapsto \hom_{\hop{k}}(S^i_s \wedge U_+,(\mathcal{X},x))$ (such sheaves are typically denoted with boldface symbols).


\section{Geometric models of motivic spheres}
\label{s:spheres}
This section is devoted to establishing Theorem \ref{thmintro:main} from the introduction.  The argument proceeds by induction.  Subsection \ref{ss:q2andq4} gives an elementary geometric argument showing that $Q_2$ is $\aone$-weakly equivalent to $\pone$, and a related (mostly elementary) argument showing that $Q_4$ is $\aone$-weakly equivalent to $({\pone})^{\sma 2}$; the first computation is the base-case for the induction argument.  Subsection \ref{ss:q2n} contains the proof of the main result, which is Theorem \ref{thm:main}.  Finally, Subsection \ref{ss:nongeometric} establishes some non-geometrizability results for motivic spheres.

\subsection{On the $\aone$-homotopy types of $Q_2$ and $Q_4$}
\label{ss:q2andq4}
\begin{prop}
\label{prop:homotopyofq2}
If $k = \Z$, then there is a (pointed) $\aone$-weak equivalence $Q_2 \isomt {\pone}$.
\end{prop}

\begin{proof}
We can identify $Q_2$ as the quotient of $SL_2$ by its maximal torus $T := diag(t,t^{-1}) \cong \gm{}$ acting by right multiplication; this follows from a straightforward computation of invariants.  If $B$ is the linear algebraic group of upper triangular matrices with determinant $1$, then the closed immersion group homomorphism $T \hookrightarrow B$ induces a morphism of homogeneous spaces $SL_2/T \to SL_2/B$.  This morphism of homogeneous spaces is Zariski locally trivial with fibers isomorphic to $\aone$ (use the standard open cover of $\pone$ by two copies of $\aone$), in particular an $\aone$-weak equivalence.  To conclude, we observe that the quotient $SL_2/B$ is isomorphic to $\pone$, and we can make this morphism a pointed morphism by picking any $k$-point in $SL_2/T$ and looking at the image of this $k$-point in $\pone$.
\end{proof}

There are two ``natural" (pointed) $\aone$-weak equivalences of $Q_2 \to \pone$ that we know.  In the previous proof, the $\aone$-weak equivalence we wrote down arose from the structure of $Q_2$ as a homogeneous space, but as we now explain this seems to be an ``exceptional" low-dimensional phenomenon.  If $k$ is a ring in which $2$ is invertible, the varieties $Q_{2n}$ are all isomorphic to homogeneous spaces for suitable orthogonal groups, but this seems false if $2$ is not invertible in $k$.

The other ``natural" (pointed) $\aone$-weak equivalence arises as follows.  The locus of points where $x_1 = 0$ and $z = -1$ is a closed subscheme of $Q_2$ isomorphic to ${\mathbb A}^1$; the open complement of this closed subscheme is isomorphic to ${\mathbb A}^2$.  The normal bundle to this embedding of $\aone$ comes equipped with a prescribed trivialization, and we can use homotopy purity together with $\aone$-contractibility of ${\mathbb A}^2$ to construct another $\aone$-weak equivalence of $Q_2$ with $\pone$.    We now explain how to identify the $\aone$-homotopy type of $Q_4$ using the method just sketched.

\begin{prop}
\label{prop:homotopyofq4}
If $k = \Z$, then there is a (pointed) $\aone$-weak equivalence $Q_4 \isomt {\pone}^{\sma 2}$.
\end{prop}

\begin{proof}
Let $E_2$ be the closed subscheme of $Q_4$ defined by $x_1 = x_2 = 0$ and $z = -1$.  Observe that $E_2$ is isomorphic to ${\mathbb A}^2$.  The normal bundle to $E_2$ is trivial, and we fix a choice of trivialization.  Let $X_4 := Q_4 \setminus E_2$.  The homotopy purity theorem gives a cofiber sequence of the form
\[
X_4 \longrightarrow Q_4 \longrightarrow Th(\nu_{E_2/Q_4}) \longrightarrow \cdots.
\]
The choice of trivialization determines an isomorphism $Th(\nu_{E_2/Q_4}) \cong {\pone}^{\sma 2} \wedge (E_2)_+$.  The map $E_2 \to \Spec k$ is an $\aone$-weak equivalence, it follows that $Th(\nu_{E_2/Q_4}) \cong {\pone}^{\sma 2}$ in $\hop{k}$.  We showed in \cite[Corollary 3.1 and Remark 3.3]{ADBundle} that $X_4$ is $\aone$-contractible over $\Spec \Z$ (indeed, it is the base of a Zariski locally trivial morphism with total space ${\mathbb A}^5$ and fibers isomorphic to $\aone$).  Properness of the $\aone$-local model structure \cite[\S 2 Theorem 3.2]{MV} guarantees that pushouts of $\aone$-weak equivalences along cofibrations are $\aone$-weak equivalences.  Applying this fact to the diagram $\ast \leftarrow X_4 \rightarrow Q_4$, one concludes that the map $Q_4 \to Th(\nu_{E_2/Q_4})$ is an $\aone$-weak equivalence and the proposition follows by combining the stated isomorphisms.
\end{proof}

\begin{rem}
\label{rem:quaternionicprojectivespace}
The variety $Q_4$ is isomorphic to the variety $\mathrm{HP}^1$ studied in \cite[\S 3]{PaninWalterPontryaginClasses}.  The variety we call $X_4$ appears in the ``cell decomposition" of $\mathrm{HP}^1$ described in \cite[Theorem 3.1]{PaninWalterPontryaginClasses} (it is the variety they call $X_0$).
\end{rem}

\subsection{On the $\aone$-homotopy type of $Q_{2n}$, $n \geq 3$}
\label{ss:q2n}
\subsubsection*{The ``octahedral" axiom in a pretriangulated category}
If $\mathcal{C}$ is a pointed model category, then recall that one can understand the homotopy cofiber of a composite in terms of the homotopy cofibers of the constituents by means of the following result, which is sometimes called the ``octahedral" axiom, by analogy with a corresponding result for stable model categories \cite[Proposition 6.3.6]{Hovey}.

\begin{prop}
\label{prop:octahedral}
Suppose given a sequence of maps in a pointed model category $\mathcal{C}$ of the form:
\[
X \stackrel{u}{\longrightarrow} Y \stackrel{v}{\longrightarrow} Z.
\]
If $U = \operatorname{hocofib}(u)$, $V = \operatorname{hocofib}(uv)$, $W = \operatorname{hocofib}(v)$, then there is a commutative diagram of the form
\[
\xymatrix{
X \ar[r]^u \ar[d]^{id} & Y \ar[r]\ar[d]^v & U \ar[d]^r & \\
X \ar[r]^{uv} \ar[d]^{u} & Z \ar[r]\ar[d]^{id} & V \ar[d]^s & \\
Y \ar[r]^v &        Z \ar[r]       & W &
}
\]
where the rows and the third column are cofiber sequences, and the maps $r$ and $s$ are the maps induced by functoriality of the homotopy cofiber construction.
\end{prop}

\subsubsection*{The induction step}
Consider the quadric $Q_{2n}$ introduced in the introduction over $\Z$.  The open subscheme $U_n$ defined by the non-vanishing of $x_n$ is isomorphic to ${\mathbb A}^{2n-1} \times \gm{}$ (with coordinates $(x_1,\ldots,x_{n-1},y_1,\ldots,y_{n-1},z)$ on the first factor and $x_n$ on the second factor).  The closed complement of this open subscheme, which we will call $Z_n$, is isomorphic to $Q_{2n-2} \times \aone$, with coordinates $x_1,\ldots,x_{n-1},y_1,\ldots,y_{n-1},z$ on the first factor and $y_n$ on the second factor.  The subvariety $Z_n$ has a distinguished point ``0" corresponding to $x_1 = \cdots = x_n = y_1 = \cdots = y_{n} = z = 0$.  Moreover, the normal bundle to $Z_n$ in $Q_{2n}$ is a line bundle equipped with a chosen trivialization (coming from $x_n$).

The closed subscheme of $U_n$ defined by $x_1 = \cdots = x_{n-1} = y_1 = \cdots = y_{n-1} = z = 0$ is isomorphic to $\gm{}$ with coordinate $x_n$ and the inclusion map
\[
\gm{} \longrightarrow U_n
\]
is a cofibration and $\aone$-weak equivalence (it even admits a retraction).  The closed subscheme of $Q_{2n}$ defined by  $x_1 = \cdots = x_{n-1} = y_{1} = \cdots = y_{n-1}$ is isomorphic to $Q_2$.  Imposing further the condition $y_n = 0$, the resulting variety is isomorphic to two copies of ${\mathbb A}^1$ (corresponding to $z = 0$ and $z = -1$).  Imposing yet further the condition $z = 0$ gives a subvariety of $Q_{2n}$ isomorphic to $\aone$ with coordinate $x_n$.  Note that, by construction, the intersection of this copy of the affine line with $Z_n$ is precisely the point ``0", i.e., there is a pullback square (of smooth schemes) of the form:
\[
\xymatrix{
0 \ar[r]\ar[d] & {\mathbb A}^1 \ar[d] \\
Z_n \ar[r] & Q_{2n}.
}
\]
Moreover, the natural map from the pullback of the normal bundle to $Z_n$ in $Q_{2n}$ to the normal bundle of $0$ in ${\mathbb A}^1$ is an isomorphism compatible with the specified trivializations.

It follows from the construction of the purity isomorphism \cite[\S 3 Theorem 2.23]{MV} that, given a pullback square as in the previous paragraph, the diagram
\[
\xymatrix{
{\mathbb A}^1/\gm{} \ar[r]\ar[d] & Th(\nu_{0/{\mathbb A}^1}) \ar[d] \\
Q_{2n}/U_n \ar[r] & Th(\nu_{Z_n/Q_{2n}})
}
\]
is commutative in $\hop{k}$ (cf. \cite[Lemma 2.1]{VMod2}).  Moreover, the specified trivializations yield $\aone$-weak equivalences of the form $Th(\nu_{0/{\mathbb A}^1}) \cong {\mathbb P}^1$ and $Th(\nu_{Z_n/Q_{2n}}) \cong {\mathbb P}^1 \sma (Z_n)_+$ by \cite[\S 3 Proposition 2.17.2]{MV} and their compatibility ensures that the diagram
\[
\xymatrix{
Th(\nu_{0/{\mathbb A}^1}) \ar[r] \ar[d] & {\mathbb P}^1 \ar[d] \\
Th(\nu_{Z_n/Q_{2n}}) \ar[r] & {\mathbb P}^1 \sma (Z_n)_+
}
\]
commutes (this time in the category of pointed spaces).  Furthermore, under these identifications, we can make the right vertical map in the diagram very explicit. Indeed, consider the map $S^0_s \to (Z_n)_+$ sending the base-point of $S^0_s$ to the base-point of $(Z_n)_+$ and the point $1$ of $S^0_s$ to the point ``$0$" in $Z_n$ described above; the right hand vertical map is the $\pone$-suspension of this map.

Now, we apply Proposition \ref{prop:octahedral} to the composition
\[
(\gm{})_+ \stackrel{u}{\longrightarrow} ({\mathbb A}^1)_+ \stackrel{v}{\longrightarrow} (Q_{2n})_+.
\]
Observe that the cofiber of $v$ is $Q_{2n}$ pointed with the image of $``0"$, while the cofiber of $u$ is, by means of the purity isomorphism, ${\mathbb P}^1$.  It remains to identify the cofiber of $(\gm{})_+ \to (Q_{2n})_+$, which is the content of the next lemma.

\begin{lem}
\label{lem:cofiberofuv}
If $k = \Z$, then the cofiber of the map $(\gm{})_+ \to (Q_{2n})_+$ is $\aone$-weakly equivalent to ${\mathbb P}^1 \sma (Q_{2n-2})_+$.
\end{lem}

\begin{proof}
The projection map $Z_n \to Q_{2n-2}$ is an $\aone$-weak equivalence, it follows that it induces an $\aone$-weak equivalence ${\mathbb P}^1 \sma (Z_n)_+ \isomt {\mathbb P}^1 \sma (Q_{2n-2})_+$.  On the other hand, by the purity isomorphism, the space ${\mathbb P}^1 \sma (Z_n)_+$ is $\aone$-weakly equivalent to $Q_{2n}/U_n$.  Since the map $\gm{} \to U_n$ described above is an $\aone$-weak equivalence and cofibration and since the induced map $\gm{} \to Q_{2n}$ is a cofibration, it follows that the induced map of cofibers
\[
Q_{2n}/\gm{} \to Q_{2n}/U_n
\]
is an $\aone$-weak equivalence.  Moreover, the same statement remains true for the cofibers of the pointed morphisms $(\gm{})_+ \to (Q_{2n})_+$ and $(U_{n})_+ \to (Q_{2n})_+$.
\end{proof}

Combining Proposition \ref{prop:octahedral} and Lemma \ref{lem:cofiberofuv}, we deduce the following result.

\begin{cor}
\label{cor:cofibersequence}
If $k = \Z$, then for any integer $n \geq 1$, there is a cofiber sequence of the form
\[
{\pone} \longrightarrow {\mathbb P}^1 \sma (Q_{2n-2})_+ \longrightarrow Q_{2n} \longrightarrow \cdots.
\]
\end{cor}

To finish, we provide an alternative identification of the cofiber of ${\pone} \to {\mathbb P}^1 \sma (Q_{2n-2})_+$.  Recall that this map is induced by $\pone$-suspension of a map $S^0_s \to (Z_n)_+$.  We prove a more general statement about smashing this map with the identity map.

\begin{prop}
\label{prop:smashproduct}
Assume $(\mathcal{X},x)$ is a pointed space, and $\mathcal{X}_+$ is $\mathcal{X}$ with a disjoint base-point attached.  Let $\iota: S^0_s \to \mathcal{X}_+$ be the map that sends the base-point of $S^0_s$ to the (disjoint) base-point of $\mathcal{X}_+$ and the non-base-point to $x \in \mathcal{X}(k)$.  If $(\mathcal{Y},y)$ is a pointed space, then the map $\mathcal{Y} \cong S^0_s \sma \mathcal{Y} \stackrel{\iota \sma id}{\longrightarrow} \mathcal{X}_+ \sma \mathcal{Y}$ fits into a split cofiber sequence of the form
\[
\mathcal{Y} \stackrel{\iota \sma id}{\longrightarrow} \mathcal{X}_+ \sma \mathcal{Y} \longrightarrow \mathcal{X} \sma \mathcal{Y}.
\]
\end{prop}

\begin{proof}
The splitting is given by the map $\mathcal{X}_+ \to S^0_s$ that collapses the non base-point component of $\mathcal{X}_+$ to the non-base-point of $S^0_s$.  In particular, the map $\iota \sma id$ is a cofibration.  The homotopy cofiber of this map is then computed by the actual quotient.

The quotient $\mathcal{X}_+ \sma \mathcal{Y}$ is, by definition, the quotient $\mathcal{X}_+ \times \mathcal{Y}$ by $\mathcal{X}_+ \vee \mathcal{Y}$, where the latter is the disjoint union of $\mathcal{X}_+ \times y$ and $\Spec k \times \mathcal{Y}$.  Identify $\mathcal{X}_+ \times \mathcal{Y}$ as $(\mathcal{X} \times \mathcal{Y}) \coprod (\Spec k \times \mathcal{Y})$.  We describe the quotient in two steps.  First, we collapse $\Spec k \times \mathcal{Y}$ to the base-point and then the image of $\mathcal{X}_+ \times y$ to the base-point.  We conclude that there is an isomorphism of spaces $\mathcal{X}_+ \sma \mathcal{Y} \cong (\mathcal{X} \times \mathcal{Y})/(\mathcal{X} \times y)$.  To conclude, we simply observe that $\mathcal{X} \times \mathcal{Y}/(\mathcal{X} \times y)$ modulo the image of $x \times \mathcal{Y}$ is $\mathcal{X} \sma \mathcal{Y}$.
\end{proof}

\begin{thm}
\label{thm:main}
For any base ring $k$, there is an isomorphism in $\hop{k}$ of the form $Q_{2n} \isomt {\pone}^{\sma n}$.
\end{thm}

\begin{proof}
Since both sides of the weak equivalence are preserved by base-change, it suffices to prove the result for $k = \Z$.  In that case, by Proposition \ref{prop:homotopyofq2}, we know that $Q_2 \isomt {\pone}$ in $\hop{k}$.  Combining Proposition \ref{prop:smashproduct} and Corollary \ref{cor:cofibersequence}, we conclude that for every $n \geq 2$ that there is an isomorphism in $\hop{k}$ of the form $Q_{2n} \cong Q_{2n-2} \sma {\pone}$.  The result follows by a straightforward induction.
\end{proof}

\subsection{Non-geometric motivic spheres}
\label{ss:nongeometric}
\begin{prop}
\label{prop:nonrepresentability}
If $k$ is a base ring, and if $i > j$ are integers, then $S^i_s \wedge \gm{\sma j}$ does not have the $\aone$-homotopy type of a smooth scheme.
\end{prop}

\begin{proof}
Assume to the contrary that $X$ is a smooth $k$-scheme having the $\aone$-homotopy type of $S^i_s \wedge \gm{\sma j}$.  By picking a a geometric point of $\Spec k$, we can, without loss of generality, assume $k$ is an algebraically closed field, in particular perfect.  Then, by $\aone$-representability of motivic cohomology (see, e.g., \cite[\S 2]{VRed}), we know that for arbitrary integers $p,q$
\[
H^{p,q}(X,\Z) = [X,K(\Z(q),p)]_{\aone}.
\]
Next, observe that by the cancellation theorem \cite[Corollary 4.10]{VCancellation} and the (simplicial) suspension isomorphism, there is a non-trivial morphism $S^i_s \wedge \gm{\wedge j} \to K(\Z(j),i+j)$, giving a non-trivial class in $H^{i+j,j}(X,\Z)$.  However, if $X$ is a smooth scheme, and $i > j$, then this latter group must vanish by \cite[Theorem 19.3]{MVW}.
\end{proof}

\begin{rem}
One can obtain another more ``homotopic" proof of the proposition at the expense of introducing slightly different terminology; we freely use the conventions of \cite{AsokFaselSpheres} and some terminology from \cite{MField}.  Consider an Eilenberg-Mac Lane space $K(\K^M_j,i)$, i.e., a space with exactly $1$ non-vanishing $\aone$-homotopy sheaf in degree $j$, in which degree it is isomorphic to the unramified Milnor K-theory sheaf $\K^M_j$.  If $X$ is a smooth scheme, the explicit Gersten resolution of $\K^M_j$ shows that $H^i(X,\K^M_j)$ vanishes for $i > j$.  On the other hand, there is a non-trivial (pointed) morphism $S^i_s \wedge \gm{\sma j} \to K(\K^M_j,i)$.  Indeed, by \cite[Theorem 6.13]{MField}, we have $\bpi_{i,j}^{\aone}(K(\K^M_j,i)) \cong (\K^M_j)_{-j}$ and by applying, e.g., \cite[Lemma 2.7]{AsokFaselSpheres} we deduce $(\K^M_j)_{-j} \cong \Z$.
\end{rem}

The following conjecture summarizes our expectations regarding representability of the remaining motivic spheres by smooth schemes.

\begin{conj}
\label{conj:nonrepresentability}
If $k$ is a base ring, and if $i,j \in {\mathbb N}$ with $i < j-1$, the sphere $S^i_s \wedge \gm{\sma j}$ does not have the $\aone$-homotopy type of a smooth $k$-scheme.
\end{conj}

\section{$\aone$-contractible subvarieties of quadrics}
\label{s:contractibles}
A smooth $k$-scheme $X$ is $\aone$-contractible if the structure map $X \to \Spec k$ is an isomorphism in $\ho{k}$, i.e., an $\aone$-weak equivalence.  The affine space ${\mathbb A}^n_k$ is $\aone$-contractible by construction of the $\aone$-homotopy category.  An {\em exotic $\aone$-contractible smooth scheme} is an $\aone$-contractible smooth scheme that is not isomorphic to an affine space.  In \cite{ADContractible} the first two authors constructed many exotic $\aone$-contractible smooth schemes as quotients of an affine space by a free action of the additive group $\ga$ (or, more generally, free actions of unipotent groups on affine spaces).  All of the examples constructed in \cite{ADContractible} could be realized as open subschemes of affine schemes with complement of codimension $\leq 2$.

Let $E_n \subset Q_{2n}$ be the closed subscheme defined by $x_1 = \cdots = x_n = 0$ and $z = -1$.  Observe that $E_n \cong {\mathbb A}^n$ and has codimension $n$ in $Q_{2n}$.  Set
\[
X_{2n} := Q_{2n} \setminus E_n.
\]
It is straightforward to check that $X_2 \cong {\mathbb A}^2$, and we recalled in the proof of Proposition \ref{prop:homotopyofq4} that $X_4$ is $\aone$-contractible.  The main results of this section are Theorem \ref{thm:newaonecontractibles}, which shows that $X_{2n}$ is $\aone$-contractible for $n \geq 1$, and Corollary \ref{cor:notaunipotentquotient}, which establishes that $X_{2n}$ is {\em not} a unipotent quotient of affine space for $n \geq 3$.

\subsection{On the $\aone$-contractibility of $X_{2n}$}
\label{ss:aonecontractibility}
If $k = \cplx$, one can check that $X_{2m}(\cplx)$ is a contractible complex manifold.  Over $\cplx$, the variety $Q_{2m}$ is isomorphic to the variety defined by $\sum_{i=1}^{2m+1} x_i^2 = 1$.  Wood explains \cite[\S 2]{WoodQuad} how to identify the last variety with the tangent bundle of a sphere.  Under this isomorphism, the variety $E_m$ can be identified with the tangent space at a point.  The complement of $E_m$ is then a contractible topological space diffeomorphic to $\real^{4m}$.  Here is the generalization of this result to $\aone$-homotopy theory.

\begin{thm}
\label{thm:newaonecontractibles}
If $k$ is a base ring, then the variety $X_{2n}$ is $\aone$-contractible over $\Spec k$.
\end{thm}

\begin{proof}
Once more, since both sides are preserved by base-change, it suffices to prove the result with $k = \Z$.  We continue with the notation of Subsection \ref{ss:q2n}.  The proof is essentially a repeat of the Proof of Theorem \ref{thm:main}, so we will explain only the changes required.  Consider the closed subvariety $Z_n \subset Q_{2n}$ defined by $x_n = 0$.  We saw that $Z_n \cong Q_{2n-2} \times \aone$ and that the normal bundle to $Z_n$ comes equipped with a specified trivialization.  Observe that, by construction $E_n \subset Z_n$ and therefore we see that $X_{2n} := Q_{2n} \setminus E_n$ contains $U_n := Q_{2n} \setminus Z_n$ as an open subscheme.  Moreover, the subvariety $Z_n \setminus E_n$ is isomorphic to $X_{2n-2} \times \aone$.

Recall that $U_n \cong {\mathbb A}^{2n-1} \times \gm{}$ with coordinates $x_1,\ldots,x_{n-1},y_1,\ldots,y_{n-1},z$ on the first factor and $x_n$ on the last factor.  The closed subscheme $\gm{} \hookrightarrow U_n$ is defined by $x_1 = \cdots = x_{n-1} = y_1 = \cdots y_{n-1} = 0$ and $z = 0$.  The intersection of $X_{2n}$ and $Q_2$, viewed as a subvariety of $Q_{2n}$ by imposing $x_1 = \cdots = x_{n-1} = y_1 = \cdots = y_{n-1} = 0$ is isomorphic to ${\mathbb A}^2$.  The copy of $\aone \subset Q_{2}$ identified by further imposing the conditions $y_n = z = 0$ is disjoint from $E_n$ and therefore contained in $X_{2n}$.  Note that the point ``0" of $Q_{2n}$ is contained in $X_{2n}$.

We therefore have a diagram of the form
\[
\gm{} \stackrel{u}{\longrightarrow} \aone \stackrel{v}{\longrightarrow} X_{2n}.
\]
We understand the cofiber of this composite map using Proposition \ref{prop:octahedral} again.  In particular, repeating the proofs of Lemma \ref{lem:cofiberofuv} and Corollary \ref{cor:cofibersequence} with all instances of $Q_{2i}$ replaced by $X_{2i}$, we deduce the existence of a cofiber sequence of the form
\[
{\pone} \stackrel{\iota}{\longrightarrow} (X_{2n-2})_{+} \sma {\pone} \longrightarrow X_{2n} \longrightarrow \cdots
\]
Now, applying Proposition \ref{prop:smashproduct}, we conclude that the cofiber of $\iota$ is $\aone$-weakly equivalent to $X_{2n-2} \sma \pone$.  The induction hypothesis guarantees that $X_{2n-2}$ is $\aone$-contractible, so the smash product $X_{2n-2} \sma \pone$ is $\aone$-contractible as well.  Since $X_{2n-2} \sma \pone$ is $\aone$-weakly equivalent to $X_{2n}$, it follows that $X_{2n}$ is $\aone$-contractible.
\end{proof}

\begin{rem}
\label{rem:alternativeproof}
Theorem \ref{thm:newaonecontractibles} gives an alternative proof of Theorem \ref{thm:main}, this time repeating the argument of Proposition \ref{prop:homotopyofq4}.  In this case, the map $Q_{2n} \to Th(\nu_{E_n/Q_{2n}})$ obtained by collapsing $X_{2n}$ to a point and then using the homotopy purity theorem is an $\aone$-weak equivalence by properness of the $\aone$-local model structure, as explained in the proof of Proposition \ref{prop:homotopyofq4}.  The explicit trivialization of the normal bundle to $E_n$ that arises by its presentation as a (connected component of a) complete intersection defined by a regular sequence of $Q_{2n}$ yields an isomorphism $Th(\nu_{E_n/Q_{2n}}) \cong {\pone}^{\sma n} \wedge (E_n)_+$, and the projection map $E_n \to \Spec k$ then yields an $\aone$-weak equivalence $Q_{2n} \isomt {\pone}^{\sma n}$ that is perhaps slightly more explicit than the one arising in the proof of Theorem \ref{thm:main}.
\end{rem}

\subsection{$\aone$-contractibles that are not unipotent quotients}
\label{ss:nonunipotentquotient}
We will now see that for $m \geq 3$, the varieties $X_{2m}$ studied in Theorem \ref{thm:newaonecontractibles} cannot be realized as quotients of affine space by free actions of unipotent groups $U$.  We begin by proving a general ``excision" style result (Theorem \ref{thm:unipotentexcision}) for (Zariski) cohomology with coefficients in a unipotent group; the main result then follows from this excision result when applied in the special case of degree $1$ cohomology (Corollary \ref{cor:codim3quotient}).  Indeed, the degree $1$ cohomology group in question can be identified as the (pointed) set of $U$-torsors on the scheme $X$.  We warn the reader that, contrary to ``understood" prior conventions, the letter $U$ in this section stands for a unipotent group, as opposed to an open subscheme; we refer the reader to \cite[\S 15 (Definition 15.1)]{Borel} for general information on split unipotent groups.

\begin{thm}[Excision for unipotent groups]
\label{thm:unipotentexcision}
Let $d \geq 3$ be an integer.  Suppose $X$ is a regular scheme over a field and $j: W \hookrightarrow X$ is an open immersion whose closed complement has codimension $\geq d$.  Suppose $U$ is a (not necessarily commutative) split unipotent group.  The restriction map
\[
j^*: H^1_{\Zar}(X,U) \longrightarrow H^1_{\Zar}(W,U)
\]
is a pointed bijection, i.e., every $U$-torsor on $W$ extends uniquely to a $U$-torsor on $X$.
\end{thm}

\begin{proof}
Since $U$ is split, by definition it admits an increasing filtration by normal subgroups $1 = U_0 \subset U_1 \subset \cdots U_n = U$ with subquotients $U_{i+1}/U_{i}$ isomorphic to $\ga$.  We first show that, working inductively with respect to the dimension of $U$, it suffices to prove the result in the case $U = \ga$.

Indeed, suppose we know the result holds true for all unipotent groups of dimension $n$ and all open immersions of regular schemes with closed complement of codimension $d \geq 3$.  In that case, suppose $U$ has dimension $n+1$ and we are given a $U$-torsor $\varphi: P \to W$.  We know that there is a normal subgroup $U'$ of $U$ such that $U/U' \cong \ga$.  In that case, the quotient map $U \to \ga$ gives rise to a $\ga$-torsor $W' := P \times^{U} \ga \to W$; note that $W'$ is again regular.  The pullback of $\varphi$ along $W' \to W$ gives a $U'$-torsor on $W'$.  Now, again by induction, we also know that the $\ga$-torsor $W' \to W$ extends (uniquely) to a $\ga$-torsor $X' \to X$; again, $X'$ is regular.  Moreover, working over a Zariski trivialization of $X' \to X$, we conclude that $W' \subset X'$ is open with closed complement having codimension $\geq 3$ as well.  However, the induction hypothesis guarantees that the $U'$-torsor over $W'$ extends uniquely to a $U'$-torsor on $X'$, and this extension provides the required extension of $\varphi$ over $W$.

Now, assume $U = \ga$.  In this case, one knows that by definition $H^q(X,\ga) = H^q(X,\O_X)$.  We will show that the map
\[
j^*: H^q(X,\O_X) \longrightarrow H^q(W,\O_W)
\]
induced by pull-back along $j$ is an isomorphism for $q \leq d-2$.

To prove the last fact, we use the Cousin complex for $\O_X$ as discussed in, e.g., \cite[IV.2]{HartshorneRD}.  Under the assumption that $X$ is regular, the Cousin complex provides an injective resolution of $\O_X$ (see especially {\em ibid.} p. 239).  If $X^{(p)}$ denote the set of codimension $p$ points in $X$ (recall this means $\dim \O_{X,x} = p$), the $p$-th term of the Cousin complex is
\[
\coprod_{x \in X^{(p)}} (i_x)_*(H^p_x(\O_X)).
\]

Since $j$ is an open immersion, it is an affine morphism and the Leray spectral sequence for $j$ degenerates to yield isomorphisms $j^*: H^q(W,\O_W) \isomt H^q(X,j_*\O_W)$. Adjunction gives rise to a morphism $\O_X \to j_* \O_W$. The push-forward of the Cousin complex for $\O_W$ to $X$ by $j$ remains a flasque resolution of $j_*\O_W$ on $X$.  Since the inclusion $W \hookrightarrow X$ is, by definition, an isomorphism on points of codimension $d-1$, it follows that the cokernel of the induced morphism of Cousin complexes, which provides a flasque resolution of the cone of the map $\O_X \to j_* \O_X$, only depends on points of codimension $\geq d$ and the isomorphism of the theorem statement follows immediately from the long exact sequence in cohomology.
\end{proof}

\begin{cor}
\label{cor:codim3quotient}
If $X$ is a regular affine scheme and $W \subset X$ is an open subscheme whose complement has codimension $d \geq 3$, then every torsor under a split unipotent group over $W$ is trivial.  In particular, the varieties $X_{2m}$ cannot be realized as quotients of affine space by free actions of split unipotent groups.
\end{cor}

\begin{proof}
By Serre's vanishing theorem \cite[Th\'eor\`eme 1.3.1]{EGAIII1}, we know that all higher cohomology of a quasi-coherent sheaf on an affine scheme vanishes.  By Theorem \ref{thm:unipotentexcision}, we conclude that if $W$ is as in the the theorem statement, and if $U$ is a (split) unipotent group, then $H^1(W,U) = \ast$, so every $U$-torsor over $W$ is trivial.  In particular, the total space of any $U$-torsor over $W$ is of the form $W \times U$, which is itself a quasi-affine variety that is not affine.
\end{proof}

\begin{cor}
\label{cor:notaunipotentquotient}
If $k$ is a base ring, and $m \geq 3$ is an integer, then $X_{2m}$ is not a quotient of an affine space by the free action of a unipotent group.
\end{cor}

\begin{proof}
Pick an algebraically closed field $F$ such that $\Spec k$ has a $\Spec F$-point.  Assume that $X_{2m}$ is a quotient of ${\mathbb A}^n$ by the free action of a unipotent group.  By base change, we conclude that the unipotent group is necessarily smooth.  Base-change to $\Spec F$ then allows us to assume that the unipotent group in question is also split.  For $m \geq 3$, $X_{2m}$ has codimension $m \geq 3$ in $Q_{2m}$, so the result then follows immediately from Corollary \ref{cor:codim3quotient}.
\end{proof}

\begin{question}
If $Z \to X_{2n}$ is a Jouanolou device, then $Z$ is an affine $\aone$-contractible variety.  Is $Z \cong {\mathbb A}^m$ for some integer $m$?
\end{question}

\begin{rem}
If $k$ is a field having characteristic $p > 0$, work of N. Gupta \cite{Gupta, Gupta2} shows that there exist smooth affine $k$-varieties of every dimension $d \geq 3$ that are stably isomorphic to affine $k$-space, yet which are not isomorphic to affine $k$-space.  Such varieties are necessarily exotic smooth affine $\aone$-contractible varieties.  At the moment, we do not have any examples of exotic smooth affine $\aone$-contractible varieties over a fields having characteristic $0$.  Nevertheless, it seems reasonable to ask: if $k$ is an arbitrary field, is every smooth affine $\aone$-contractible $k$-variety a retract of an affine $k$-space?
\end{rem}

\section{Clutching functions for principal $G$-bundles on algebraic spheres}
\label{s:vectorbundles}
In this section, we provide an algebro-geometric analog of the clutching construction of vector bundles on even-dimensional spheres. As an application, we provide an explicit {\em vector bundle} yielding a generator of $\widetilde{K}_0(Q_{2n})$ for $n\geq 1$.  For $n = 1,2$, such explicit generators are given by classical geometric constructions (Hopf bundles).  For $n \geq 3$, while it is easy to write down virtual vector bundle representatives, writing down explicit vector bundles is more complicated.

The main results of this section are Theorem \ref{thm:abstractclutching}, which provides a version of the clutching construction for any smooth affine scheme $Y$ whose (simplicial) suspension has the $\aone$-homotopy type of a smooth affine scheme.  Making the abstract clutching construction explicit depends on the choice of weak equivalence $\Sigma^1_s Y$ with an affine scheme $X$.  In the special case where $Y = Q_{2n-1}$ in Theorem \ref{thm:explicitequivalence}, we construct an explicit weak equivalence $\Sigma^1_s Q_{2n-1} \to Q_{2n}$, at least under suitable additional hypotheses.  Finally in Theorem \ref{thm:concreteclutching}, we combine these two observations to obtain a version of the clutching construction that gives explicit cocycle representatives of vector bundles.

\subsection{The abstract clutching construction}
We begin by providing an abstract analog of the clutching construction in algebraic geometry.  Let us begin by quickly reviewing the topological situation, where the clutching construction is sometimes phrased as follows.  Suppose $M$ is a pointed finite CW complex, and consider the (reduced) suspension $\Sigma M$.  The usual identification $G \cong \Omega BG$ together with the loop-suspension adjunction yield a canonical bijection between sets of pointed homotopy classes of maps of the form $[M,G]_\bullet \cong [\Sigma M,BG]_{\bullet}$.  By the representability theorem for principal $G$-bundles, the forgetful map $[\Sigma M,BG]_{\bullet} \to [\Sigma M,BG]$ from based to free homotopy classes yields a surjection from $[\Sigma M,BG]_{\bullet}$ to the set of isomorphism classes of principal $G$-bundles on $\Sigma M$.  If $G$ is furthermore connected, then $BG$ is necessarily $1$-connected, then the set of pointed homotopy classes of maps $[M,G]_{\bullet}$ is actually in bijection with the set of isomorphism classes of principal $G$-bundles on $\Sigma M$.

By Theorem \ref{thm:main}, we find ourselves in an analogous situation.  Indeed, we know that $Q_{2n} \cong {\pone}^{\sma n}$ in $\hop{k}$ for $n \geq 0$.  Using the standard identification $Q_{2n-1} \cong \Sigma^{n-1}_s \gm{\sma n}$, we can then conclude that $Q_{2n}$ is $\aone$-weakly equivalent to $\Sigma^1_s Q_{2n-1}$ if $n \geq 1$, though providing an explicit $\aone$-weak equivalence of this form is more complicated.  More abstractly, we are given a (pointed) smooth affine scheme $Y$ such that $\Sigma^1_s Y$ has the $\aone$-homotopy type of a smooth affine scheme.  (If $k$ is a regular ring, using Jouanolou-Thomason devices, it suffices to assume that $\Sigma^1_s Y$ has the $\aone$-homotopy type of a smooth scheme).

We would like to link principal $G$-bundles on a smooth affine model of $\Sigma^1_s Y$ with ``clutching functions" on $Y$ itself.  As $\Sigma^1_s Y$ only has the $\aone$-homotopy type of a smooth affine scheme, we appeal to the results of \cite{AHW,AHWII} on $\aone$-representability results for principal $G$-bundles (which themselves extend results of F. Morel \cite{MField} and M. Schlichting \cite{Schlichting}).  We break the argument into three pieces along the lines of what is sketched above: (i) adjunction, to obtain a bijection of $\aone$-homotopy classes of maps, (ii) connectivity statements, to obtain a suitable surjectivity statement, and (iii) representability, to connect $\aone$-homotopy classes with actual geometric objects (i.e., morphisms on the one side and bundles on the other).  These three pieces are put together in the main result of this section, Theorem \ref{thm:abstractclutching}.

We begin by recalling some notation.  Recall that a presheaf $\mathscr{F}$ on $\Sm_k$ is called {\em $\aone$-invariant on affines} if, for any smooth affine $k$-scheme $X$, the map $\F(X) \to \F(X \times \aone)$ induced by pullback along the projection $X \times \aone \to X$ is a bijection.  Write $\Sing^{\aone}(-)$ for the singular construction of \cite[p. 87]{MV}; briefly, this space is the diagonal of the bisimplicial object obtained by considering the internal hom $\hom(\Delta^{\bullet}_k,\mathscr{X})$.  Write $R_{\Zar}$ for a Zariski fibrant replacement functor on the category of simplicial presheaves.  With this notation, we can recall the results from \cite{AHW,AHWII} that we need.

\begin{lem}
\label{lem:aonelocalitystatements}
Suppose $k$ is a base ring, $G$ is a smooth $k$-group scheme and assume that the presheaf $H^1_{\Nis}(-,G)$ on $\Sm_k$ is $\aone$-invariant on affines.
\begin{enumerate}[noitemsep,topsep=1pt]
\item The spaces $R_{\Zar} \Sing^{\aone} BG$ and $R_{\Zar} \Sing^{\aone} G$ are Nisnevich local and $\aone$-invariant.
\item For any smooth affine $k$-algebra $A$, the maps $\Sing^{\aone}BG(A) \to R_{\Zar} \Sing^{\aone} BG(A)$ and $\Sing^{\aone}G(A) \to R_{\Zar}\Sing^{\aone}G(A)$ are weak equivalences.
\item For any smooth affine $k$-algebra $A$, the canonical map $\Sing^{\aone} \mathbf{R} \Omega^1_s BG(A) \to {\mathbf R}\Omega^1_s \Sing^{\aone} BG(A)$ is a weak equivalence.
\end{enumerate}
\end{lem}

\begin{proof}
The first two statements are contained in \cite[Theorems 2.2.5 and 2.3.2]{AHWII}.  The third statement follows from \cite[Corollary 2.1.2]{AHWII}.
\end{proof}

We begin by establishing the ``adjunction" part of the clutching construction.

\begin{lem}
\label{lem:clutchingadjunction}
Suppose $k$ is a base ring, $G$ is a smooth $k$-group scheme and assume that the presheaf $H^1_{\Nis}(-,G)$ on $\Sm_k$ is $\aone$-invariant on affines.  If $(Y,y)$ is a (pointed) smooth affine scheme such that $\Sigma^1_s Y$ has the $\aone$-homotopy type of a smooth affine scheme, then for any (pointed) smooth affine model $(X,x)$ of $\Sigma^1_s Y$ there is a pointed bijection
\[
[(Y,y),(G,1)]_{\aone} \longrightarrow [(X,x),(BG,\ast)]_{\aone}.
\]
\end{lem}

\begin{proof}
By Lemma \ref{lem:aonelocalitystatements}(1), we know that $R_{\Zar}\Sing^{\aone}G$ is Nisnevich local and $\aone$-invariant, i.e., $\aone$-fibrant \cite[\S 2 Proposition 3.19]{MV}.  By Lemma \ref{lem:aonelocalitystatements}(2), we know that the map $\Sing^{\aone}G \to R_{\Zar}\Sing^{\aone}G$ is a Nisnevich local weak equivalence.  In particular, we know that $\Sing^{\aone}(G)$ is $\aone$-local and thus that $[(Y,y),(G,1)]_{\aone} \cong [(Y,y),(\Sing^{\aone}(G),1)]_s$.

The evident map $G \to \mathbf{R} \Omega^1_s BG$ is a Nisnevich local weak equivalence, e.g., by \cite[Lemma 2.2.2(iii)]{AHWII}.  It follows that the induced map $\Sing^{\aone} G \to \Sing^{\aone} {\mathbf R} \Omega^1_s BG$, which can be written as a homotopy colimit of Nisnevich local weak equivalences, is again a Nisnevich local weak equivalence.  On the other hand, Lemma \ref{lem:aonelocalitystatements}(3) allows us to conclude that the map $\Sing^{\aone} \mathbf{R} \Omega^1_s BG \to {\mathbf R}\Omega^1_s \Sing^{\aone} BG$ is a Nisnevich local weak equivalence.  Combining these two observations, we deduce that the induced map $\Sing^{\aone} G \to {\mathbf R}\Omega^1_s \Sing^{\aone} BG$ is a Nisnevich local weak equivalence.  Together with the conclusion of the previous paragraph, we deduce that
\[
[(Y,y),(\Sing^{\aone}(G),1)]_s \longrightarrow [(Y,y),{\mathbf R}\Omega^1_s \Sing^{\aone} BG]_s
\]
is a bijection.

By the using the loop-suspension adjunction, we conclude that $[(Y,y),{\mathbf R}\Omega^1_s \Sing^{\aone} BG]_s \cong [\Sigma^1_s(Y,y),\Sing^{\aone} BG]_s$. However, arguing as in the first paragraph, we also know that $\Sing^{\aone} BG$ is $\aone$-local, i.e., $[\Sigma^1_s(Y,y),\Sing^{\aone} BG]_s \cong [\Sigma^1_s(Y,y),\Sing^{\aone} BG]_{\aone} = [\Sigma^1_s(Y,y),BG]_{\aone}$.  Thus, combining all these identifications, we conclude that $[(Y,y),(G,1)]_{\aone} \isomt [\Sigma^1_s(Y,y),BG]_{\aone}$.  Now, since $\Sigma^1_s Y$ has the $\aone$-homotopy type of a smooth scheme $X$ (and explicitly writing base-points for $BG$), we conclude that
\[
[(Y,y),(G,1)]_{\aone} = [(X,x),(BG,\ast)]_{\aone},
\]
which is precisely what we wanted to show.
\end{proof}

Next, we establish the analog of the surjectivity and bijectivity statements in the topological clutching construction.

\begin{lem}
\label{lem:connectivity}
Suppose $k$ is a base ring, $G$ is a smooth $k$-group scheme and assume that the presheaf $H^1_{\Nis}(-,G)$ on $\Sm_k$ is $\aone$-invariant on affines.  If $(X,x)$ is a pointed smooth affine scheme, then the evident map $[(X,x),(BG,\ast)]_{\aone} \to [X,BG]_{\aone}$ is surjective; if $G$ is furthermore $\aone$-connected, then this map is bijectictive.
\end{lem}

\begin{proof}
As in the Proof of Lemma \ref{lem:clutchingadjunction} we know that $\Sing^{\aone}G$ and $\Sing^{\aone}BG$ are $\aone$-local and that $\Sing^{\aone} G \to {\mathbf R}\Omega^1_s \Sing^{\aone} BG$ is a Nisnevich local weak equivalence.  Therefore, we conclude there are identifications of homotopy sheaves $\bpi_0^{\aone}(G) = \bpi_0(\Sing^{\aone} G) \cong \bpi_0({\mathbf R}\Omega^1_s \Sing^{\aone} BG) \cong \bpi_1(\Sing^{\aone} BG) = \bpi_1^{\aone}(BG)$.

The space $BG$ is always $0$-connected by construction, and therefore by \cite[\S 2 Corollary 3.22]{MV} it is always $\aone$-connected.  In particular, $\Sing^{\aone} BG$ is always $0$-connected.  Since $X$ is a pointed smooth scheme,
any morphism $X \to R_{\Zar} \Sing^{\aone} BG$ can be pointed and this yields the surjectivity statement.  If $G$ is $\aone$-connected, then we similarly conclude that $BG$ is $\aone$-$1$-connected.  In that case, the map of the theorem statement is a bijection by \cite[Lemma 2.1]{AsokFaselSpheres}.
\end{proof}

\begin{rem}
The space $BGL_n$ is {\em not} $\aone$-$1$-connected in general, and thus the map from pointed to free homotopy classes is {\em not} a bijection in general; see \cite[Corollary 4.12]{AsokFaselSpheres} for an explicit example to see that injectivity can indeed fail in this case.  On the other hand, the set of pointed homotopy classes of maps $[(X,x),(BG,\ast)]_{\aone}$ also has a ``concrete" description in terms of pairs consisting of a $G$-torsor $\mathscr{P}$ on $X$ equipped with a trivialization of $x^*\mathscr{P}$; we will not need this description.
\end{rem}

Now, we give a concrete description of elements of $[(Y,y),(G,1)]_{\aone}$ in terms of pointed morphisms of smooth schemes $Y \to G$.  Suppose $Y = \Spec A$, where $A$ is a $k$-algebra, a choice of $k$-point $y$ of $Y$ corresponds to a homomorphism $y: A \to k$.  This homomorphism induces a surjective morphism $res_y$ of simplicial groups $res_y: G(A[\Delta^{\bullet}]) \to G(k[\Delta^{\bullet}])$, and we set
\[
\Sing^{\aone}G(Y,y) := \ker(res_y): G(A[\Delta^{\bullet}]) \longrightarrow G(k[\Delta^{\bullet}]).
\]
Equivalently, $\Sing^{\aone}G(Y,y)$ is the (pointed) subsimplicial set of $\Sing^{\aone}G(Y)$ consisting of morphisms $Y \times \Delta^n \to G$ whose restriction to $\{ y \} \times \Delta^n$ are constant with image $1 \in G$, which explains the choice of terminology.

\begin{lem}
\label{lem:representability}
Suppose $k$ is a base ring, $G$ is a smooth $k$-group scheme and assume that the presheaf $H^1_{\Nis}(-,G)$ on $\Sm_k$ is $\aone$-invariant on affines.  Assume $(Y,y)$ is a (pointed) smooth affine scheme.
\begin{enumerate}[noitemsep,topsep=1pt]
\item The map $\pi_0(\Sing^{\aone}G(Y,y)) \to [(Y,y),(G,1)]_{\aone}$ is a bijection, i.e., $[(Y,y),(G,1)]_{\aone}$ coincides with the set of naive homotopy classes of pointed maps from $Y$ to $G$.
\item There is a pointed bijection $[Y,BG]_{\aone} \cong H^1_{\Nis}(Y,G)$.
\end{enumerate}
\end{lem}

\begin{proof}
For Point (1), we proceed as follows.  The map $\Sing^{\aone}G \to R_{\Zar}\Sing^{\aone}G$ is a morphism of pointed simplicial presheaves and upon evaluation at any smooth affine scheme $Y$ is a weak equivalence.  We can assume that $R_{\Zar}$ commutes with finite limits (it has a model given by the ``Godement resolution" \cite[\S 2 Theorem 1.66]{MV}).  Therefore, $R_{\Zar}\Sing^{\aone}G$ is again a simplicial group and $R_{\Zar}res_y: R_{\Zar}\Sing^{\aone}G(Y) \to R_{\Zar}\Sing^{\aone}G(k)$ is a homomorphism; this homomorphism is surjective because $res_y$ admits a set-theoretic section corresponding to the structure morphism.  We set $F_y := \ker(R_{\Zar}res_y)$.

By functoriality, we obtain the following commutative diagram:
\[
\xymatrix{
\Sing^{\aone}G(Y,y) \ar[r]\ar[d] & \Sing^{\aone}G(Y) \ar[d]\ar[r]^-{res_y}  & \Sing^{\aone}G(k)\ar[d] \\
F_y \ar[r]& R_{\Zar}\Sing^{\aone}G(Y) \ar[r]^-{R_{\Zar} res_y} & R_{\Zar}\Sing^{\aone}G(k).
}
\]
The two right vertical maps in this diagram are weak equivalences by Lemma \ref{lem:aonelocalitystatements}(2).  On the other hand, as $res_y$ and $R_{\Zar} res_y$ are both surjective group homomorphism, they are fibrations, with fibers isomorphic to the respective kernels, i.e., the diagram is a morphism of homotopy fiber sequences \cite[Lemma 3.2]{Curtis}.  Therefore, we conclude that the leftmost vertical map $\Sing^{\aone}G(Y,y) \to F_y$ is also a (pointed) weak equivalence.

On the other hand, $\pi_0(F_y)$ coincides with the quotient of the group of pointed morphisms $Y \to R_{\Zar}\Sing^{\aone}G$ by the homotopy relation.  Then, since $R_{Zar}\Sing^{\aone}G$ is Nisnevich local and $\aone$-fibrant by Lemma \ref{lem:aonelocalitystatements}(1), we conclude that $[(Y,y),(G,1)]_{\aone}$ coincides with $\pi_0(F_y)$.  By the results of the previous paragraphs, we conclude that $\pi_0(F_y) \cong \pi_0(\Sing^{\aone}G(Y,y))$, which is precisely what we wanted to show.

Point (2) follows immediately from \cite[Theorem 2.2.5(ii)]{AHWII}.
\end{proof}

The simplicial weak equivalence $G \cong {\mathbf R}\Omega^1_s BG$ yields, by adjunction, a canonical morphism $\Sigma^1_s G \to BG$.  If $(Y,y)$ is a pointed smooth scheme, a pointed morphism $f: Y \to G$ then determines a morphism $\widetilde{cl}_G(f): \Sigma^1_s Y \to \Sigma^1_s G \to BG$.  If $\Sigma^1_s Y$ has the $\aone$-homotopy of a smooth scheme $X$, then fixing an $\aone$-weak equivalence $\Sigma^1_s Y \cong X$, the element $\widetilde{cl}_G(f)$ determines an element of $[X,BG]_{\aone}$ by forgetting the base-point.  The resulting assignment factors through naive $\aone$-weak equivalence and determines a ``clutching" function
\[
cl_G: \pi_0(\Sing^{\aone}G(Y,y)) \longrightarrow [X,BG]_{\aone}.
\]
Regarding this clutching function, we have the following result, which puts everything above together.

\begin{thm}
\label{thm:abstractclutching}
Suppose $k$ is a base ring, $G$ is a smooth $k$-group scheme and assume that the presheaf $H^1_{\Nis}(-,G)$ on $\Sm_k$ is $\aone$-invariant on affines.  If $Y$ is a (pointed) smooth affine scheme such that $\Sigma^1_s Y$ has the $\aone$-homotopy type of a smooth affine scheme, then the assignment $f \mapsto cl_G(f)$ yields a surjective ``clutching" function of the form
\[
cl_G: \pi_0(\Sing^{\aone}G(Y,y)) \longrightarrow H^1_{\Nis}(\Sigma^1_s Y,G);
\]
this assignment is functorial in both $G$ and $Y$, and bijective if $G$ is $\aone$-connected.  Moreover, the $\aone$-invariance on affines hypothesis is satisified if:
\begin{itemize}[noitemsep,topsep=1pt]
\item[(i)] $G$ is $GL_n, SL_n$ or $Sp_{2n}$, and $k$ is smooth over a Dedekind domain with perfect residue fields, or
\item[(ii)] $k$ is an infinite field, and $G$ is an isotropic reductive $k$-group (in the sense of \textup{\cite[Definition 3.3.4]{AHWII}});
\end{itemize}
and the $\aone$-connectedness hypothesis is satisfied if, e.g., $G = SL_n$ or $Sp_{2n}$.
\end{thm}

\begin{proof}
For the first statement, simply combine Lemmas \ref{lem:clutchingadjunction}, \ref{lem:connectivity} and \ref{lem:representability}.  Points (i) and (ii) are a consequence of \cite[Theorem 5.2.1]{AHW} and \cite[Theorems 3.3.1, 3.3.2 and 3.3.6]{AHW}.  The $\aone$-connectedness of $SL_n$ or $Sp_{2n}$ if $k$ is infinite follows from existence of elementary matrix factorizations combined with \cite[Lemma 6.1.3]{MStable}.
\end{proof}

\begin{rem}
One criticism of Theorem \ref{thm:abstractclutching} is that, given an explicit pointed map $f: Y \to G$, we have only ``abstractly" produced the bundle $cl_{G}(f)$ as an $\aone$-homotopy class of maps.  While we will see this is sufficient to make interesting existence statements, to ``concretely" produce a bundle, say in terms of a suitable $G$-valued $1$-cocycle on a smooth affine model $X$ of $\Sigma^1_s Y$, requires specifying an $\aone$-weak equivalence $\Sigma^1_s Y \cong X$.
\end{rem}

\subsection{Mayer-Vietoris cofiber sequences and an explicit weak equivalence}
We now turn our attention to producing an {\em explicit} $\aone$-weak equivalence $Q_{2n} \isomt \Sigma^1_s Q_{2n-1}$; this is accomplished in Theorem \ref{thm:explicitequivalence} below.  We begin with some preliminaries on Mayer-Vietoris cofiber sequences.
Suppose $(\mathscr{X},x)$ and $(\mathscr{Y},y)$ are pointed spaces.  Write $C(\mathscr{X})$ for the usual cone on $\mathscr{X}$, i.e., $\mathscr{X} \sma \Delta^1_s$, where the latter is pointed by $1$.  If $f: \mathscr{X} \to \mathscr{Y}$, the map collapsing $\mathscr{Y}$ to a point yields a morphism of diagrams:
\[
\xymatrix{
C(\mathscr{X}) \ar[d]^-{=}& \ar[l]\ar[r]^-f\ar[d]^-{=} \mathscr{X} & \mathscr{Y} \ar[d] \\
C(\mathscr{X}) & \ar[l]\ar[r] \mathscr{X} & \ast.
}
\]
The induced map of homotopy pushouts along the rows yields a morphism $C(f) \longrightarrow \Sigma^1_s\mathscr{X}$, functorial in all the inputs.

Now, suppose we are given a smooth scheme $X$ and a Nisnevich distinguished square of the form
\[
\xymatrix{
U \times_X V \ar[r]^-{j'} \ar[d] & V \ar[d]^{\psi} \\
U \ar[r]_j & X.
}
\]
Let us assume that $(U \times_X V)(k)$ is non-empty and fix a base-point; we point $U$, $V$ and $X$ by the image of this base-point under the various maps.  In that case, the morphism $\bar{\psi}: V/(U \times_X V) \to X/U$ induced by $\psi$ is an isomorphism by \cite[\S 3 Lemma 1.6]{MV}.  On the other hand, since all horizontal morphism in the diagram are cofibrations, the quotients $V/(U \times_X V)$ and $X/U$ are models for the homotopy cofibers of $j$ and $j'$.  Thus, $\overline{\psi}^{-1}$ provides a weak equivalence $C(j) \to C(j')$.  On the other hand, there is a map $C(j') \to \Sigma^1_s(U \times_X V)$ by the discussion of the previous paragraph.  Therefore, one obtains a composite morphism of the form
\[
X \longrightarrow C(j') \isomto C(j) \longrightarrow \Sigma^1_s(U \times_X V);
\]
we use this connecting morphism below.

\begin{ex}
\label{ex:mayervietorismorphism}
If $X$ is a smooth scheme, and $X = U \cup V$ is a Zariski cover of $X$ by two open subschemes such that $(U \cap V)(k)$ is non-empty, then the construction just described yields a morphism
\[
X \longrightarrow \Sigma^1_s(U \cap V).
\]
As we will see below, the assumption $(U \cap V)(k) \neq \emptyset$ is a non-trivial restriction.
\end{ex}

Let us now specialize the constructions just made taking $X=Q_{2n}$.  Before proceeding, recall that the scheme $Q_{2n-1}$ has a standard base-point $(1,0,\ldots,0,1,0,\ldots,0)$.  Consider the two principal open subschemes $V^0_{2n}:=D_z$ and $V^1_{2n}:=D_{1+z}$.  The intersection $V^0_{2n}\cap V^1_{2n}$ is the principal open subscheme $D_{z(1+z)}\subset Q_{2n}$ on which the $2n$ sections
\[
(x_1/z,\ldots,x_n/z,y_1/(1+z),\ldots,y_n/(1+z))
\]
satsify the equation defining the quadric $Q_{2n-1}$.  We thus obtain a morphism $\psi_n: V^0_{2n}\cap V^1_{2n}\to Q_{2n-1}$.

\begin{rem}
The set $(V^0_{2n}\cap V^1_{2n})(k)$ {\em can} be empty if, e.g., $k = {\mathbb F}_2$.
\end{rem}

If $(V^0_{2n}\cap V^1_{2n})(k)$ is non-empty, then we can pick a base-point here.  Up to post-composing with an automorphism of $Q_{2n-1}$, we can assume that this base-point is mapped to the standard base-point of $Q_{2n-1}$.  Therefore, we obtain a morphism $\Sigma^1_s(V^0_{2n}\cap V^1_{2n}) \longrightarrow \Sigma^1_s Q_{2n-1}$.  Then, the connecting homomorphism in the Mayer-Vietoris cofiber sequence as in Example \ref{ex:mayervietorismorphism} composed with the morphism of the previous sentence yields a morphism
\[
\varphi_n: Q_{2n} \longrightarrow \Sigma^1_s(V^0_{2n}\cap V^1_{2n}) \longrightarrow \Sigma^1_s Q_{2n-1}.
\]
Regarding this composite morphism, we have the following result:

\begin{thm}
\label{thm:explicitequivalence}
If the base $k$ is an (infinite) perfect field, then, for any integer $n \geq 2$, the morphism $\varphi_n:Q_{2n}\to \Sigma^1_s Q_{2n-1}$ is an $\aone$-weak-equivalence.
\end{thm}

\begin{rem}
The careful reader will note that the results of Morel to which we appeal only require $k$ to be perfect.  However, the precise results to which we appeal, specifically \cite[Theorem 5.46]{MField} and \cite[Theorem 6.1]{MField} implicitly make use of \cite[Lemma 1.15]{MField}, which is a form of a presentation lemma of Gabber.   Morel does not provide a proof for this statement; the published version of Gabber's presentation lemma \cite{CTHK} requires that base field is infinite and so we have made this assumption as well.
\end{rem}

\begin{proof}
In the proof, we appeal to the results of $\cite{MField}$ and this is the reason we require the base to be an (infinite) perfect field.  The assumption that $k$ is infinite also guarantees that $(V^0_{2n}\cap V^1_{2n})(k)$ is non-empty.  By Theorem \ref{thm:main}, we have an explicit weak-equivalence $Q_{2n}\to (\pone)^{\sma n}$.  In particular, by \cite[Corollary 6.38]{MField}, $Q_{2n}$ is $\aone$-$(n-1)$-connected.  On the other hand, we know that $Q_{2n-1}\simeq S^{n-1}\wedge \gm{\sma n}$ by \cite[\S 3 Example 2.20]{MV}.  Since $n \geq 2$, $\Sigma^1_s Q_{2n-1}$ is at least $\aone$-$1$-connected, and it follows from \cite[Lemma 2.1]{AsokFaselSpheres} that for any pointed space $(\mathscr{X},x)$ the map $[(\mathscr{X},x),\Sigma^1_s Q_{2n-1}]_{\aone} \to [\mathscr{X},\Sigma^1_s Q_{2n-1}]$ is a bijection (this bijection allows us to be sloppy with pointed and free homotopy classes below).  Combining this observation with the conclusion of \cite[Corollary 6.43]{MField}, which computes $[(Q_{2n},0),\Sigma^1_s Q_{2n-1}]_{\aone}$, we conclude that $[Q_{2n},\Sigma^1_s Q_{2n-1}] \cong \K^{MW}_0(k)$.  Then, it suffices then to show that the class of the morphism $\varphi_n$ corresponds to an invertible element of $\K^{MW}_0(k)$ under the previous identifications.

Observe that any morphism $f:Q_{2n}\to \Sigma^1_s Q_{2n-1}$ yields a function
\[
[\Sigma^1_s Q_{2n-1},Q_{2n}]_{\aone} \longrightarrow [Q_{2n},Q_{2n}]_{\aone}
\]
by precomposing with $f$. Identifying both terms with $\K^{MW}_0(k)$, this map is precisely the multiplication by the class of $f$ in $[Q_{2n},\Sigma^1_s Q_{2n-1}]_{\aone}=\K^{MW}_0(k)$.

Next, we claim that $[Q_{2n},Q_{2n}]_{\aone} \isomt H^n_{\Nis}(Q_{2n},\K_n^{MW})$.  Indeed, we know $Q_{2n}$ is $\aone$-$(n-1)$-connected and $\bpi_n^{\aone}(Q_{2n}) \cong \K^{MW}_n$ by \cite[Theorem 6.40]{MField}.  Therefore, we conclude that there is an $\aone$-weak equivalence $Q_{2n}^{(n)} \cong K(\K^{MW}_n,n)$ between the $n$-th stage of the $\aone$-Postnikov tower (see \cite[Theorem 6.1]{AsokFaselThreefolds}) of $Q_{2n}$, and the Eilenberg-Mac Lane space $K(\K^{MW}_n,n)$.  The functorial morphism $Q_{2n} \to Q_{2n}^{(n)}$ thus yields a homomorphism
\[
[Q_{2n},Q_{2n}]_{\aone} \longrightarrow H^n_{\Nis}(Q_{2n},\K_n^{MW}).
\]
By \cite[Lemma 4.5]{AsokFaselSpheres} and the suspension isomorphism, for any strictly $\aone$-invariant sheaf $\mathbf{A}$ there are isomorphisms of the form
\[
H^i_{\Nis}(Q_{2n},\mathbf{A}) \cong H^i_{\Nis}(\Sigma^1_s Q_{2n-1},\mathbf{A}) \cong H^{i-1}_{\Nis}(Q_{2n-1},\mathbf{A}) = 0
\]
if $i > n$.  Thus, arguing as in \cite[Proposition 6.2]{AsokFaselThreefolds}, this vanishing implies that the map $[Q_{2n},Q_{2n}]_{\aone} \to [Q_{2n},Q_{2n}^{(n)}]_{\aone} \cong [Q_{2n},K(\K^{MW}_n,n)]_{\aone}$ is a bijection.  Reading through the above, one has also deduced that $[\Sigma^1_s Q_{2n-1},Q_{2n}]_{\aone} \isomt H^{n-1}_{\Nis}(Q_{2n-1},\K_n^{MW})$.

To establish the theorem, it suffices to prove that $\varphi_n:Q_{2n}\to \Sigma^1_s Q_{2n-1}$ yields an isomorphism
\[
H^{n-1}_{\Nis}(Q_{2n-1},\K_n^{MW}) \longrightarrow H^n_{\Nis}(Q_{2n},\K_n^{MW}).
\]
Now, by construction, the morphism induced by $\varphi_n$ is the composite
\[
H^{n-1}_{\Nis}(Q_{2n-1},\K_n^{MW}) \longrightarrow H^{n-1}_{\Nis}(V_{2n}^0\cap V_{2n}^1,\K_n^{MW}) \longrightarrow H^n_{\Nis}(Q_{2n},\K_n^{MW})
\]
where the first morphism is the pull-back along the map $\varphi_n: V^0_{2n}\cap V^1_{2n}\to Q_{2n-1}$ and the second morphism is the connecting homomorphism in the Mayer-Vietoris exact sequence. The result then rests on the explicit cohomological computations performed in Lemmas \ref{lem:generator1}, \ref{lem:generator2} and \ref{lem:composite}.
\end{proof}

We recall a few consequences of \cite[Corollary 5.43]{MField} concerning cohomology with coefficients in the sheaves $\K^{MW}_n$, which will be helpful in performing the necessary computations.  For any smooth scheme $X$, the canonical map $H^i_{\Zar}(X,\K^{MW}_n) \isomt H^i_{\Nis}(X,\K^{MW}_n)$ is an isomorphism, and the groups $H^i_{\Zar}(X,\K_n^{MW})$ can be computed using an explicit flasque resolution; the degree $m$ term of this complex takes the form
\[
\bigoplus_{x\in X^{(m)}} (\mathbf{K}^{MW}_{n-m}(k(x))\otimes_{\Z[k(x)^\times]}\Z[\Lambda_x^*])
\]
where $\Lambda_x=\wedge^m \mathfrak m_x/\mathfrak m_x^2$, $\mathfrak m_x$ is the maximal ideal in the local ring $\O_{X,x}$ and $\Lambda_x^*$ is the set of nonzero elements in the (one dimensional) $k(x)$-vector space $\Lambda_x$ (usually, one considers the dual of $\mathfrak m_x/\mathfrak m_x^2$ but this makes no difference in our analysis \cite[\S 5.1]{MField} and we drop it to lighten the notation).  In particular, in the lemmas below, cohomology could be taken either with respect to the Zariski or Nisnevich topologies.

\begin{lem}
\label{lem:generator1}
For any $n\geq 1$, the group $H^{n-1}_{\Zar}(Q_{2n-1},\mathbf{K}_n^{MW})$ is a free $\mathbf{K}_0^{MW}(k)$-module of rank one generated by the class of the cycle $[x_n]\otimes \overline {x_1}\wedge\ldots\wedge \overline x_{n-1}$ in $\mathbf{K}_1^{MW}(k(s))\otimes_{\Z[k(s)^\times]}\Z[\Lambda_s^*])$ where $s$ is the complete intersection given by the equations $x_1=\ldots=x_{n-1}=0$ (and $s$ is the generic point if $n=1$).
\end{lem}

\begin{proof}
See \cite[\S 3.3]{Fasel08b}.
\end{proof}

\begin{lem}
\label{lem:generator2}
For any $n\geq 1$, the group $H^n_{\Zar}(Q_{2n},\mathbf{K}_n^{MW})$ is a free $\mathbf{K}_0^{MW}(k)$-module of rank one generated by the class of the cycle $\langle 1\rangle\otimes \overline x_1\wedge\ldots\wedge \overline x_n$ in $\mathbf{K}_0^{MW}(k(t))\otimes_{\Z[k(t)^\times]}\Z[\Lambda_t^*])$ where $t$ is the closed subscheme defined by the vanishing locus of the sections $x_1,\ldots,x_n,1+z$.
\end{lem}

\begin{proof}
By Theorem \ref{thm:newaonecontractibles}, we know that the open subscheme $X_{2n}\subset Q_{2n}$ is $\aone$-contractible and has closed complement $E_n\simeq \mathbb{A}^n$. It suffices then to use the long exact sequence associated with this embedding, together with homotopy invariance of cohomology with coefficients in Milnor-Witt K-theory sheaves to obtain the result.
\end{proof}

\begin{lem}
\label{lem:composite}
The composite
\[
H^{n-1}_{\Zar}(Q_{2n-1},\K_n^{MW}) \longrightarrow H^{n-1}_{\Zar}(V_{2n}^0\cap V_{2n}^1,\K_n^{MW})\longrightarrow H^n_{\Zar}(Q_{2n},\K_n^{MW})
\]
induced by $\varphi_n$ maps the generator of \textup{Lemma \ref{lem:generator1}} to $\langle (-1)^n\rangle$ times the generator of \textup{Lemma \ref{lem:generator2}}.
\end{lem}

\begin{proof}
Pulling-back the generator of Lemma \ref{lem:generator1} to $V_{2n}^0\cap V_{2n}^1$ we find the cycle $[x_n/z]\otimes \overline {x_1/z}\wedge\ldots\wedge \overline x_{n-1}/z$. Its image under the boundary map $H^{n-1}_{\Zar}(V_{2n}^0\cap V_{2n}^1,\K_n^{MW})\to H^n_{1+z}(V_{2n}^0,\K_n^{MW})$ is the cycle $\langle 1\rangle\otimes \overline x_1/z\wedge\ldots\wedge \overline x_n/z$ in $\mathbf{K}_0^{MW}(k(t))\otimes_{\Z[k(t)^\times]}\Z[\Lambda_t^*])$. As $z=-1$ in $k(t)$, the above cycle corresponds to $\langle (-1)^n\rangle\otimes \overline x_1\wedge\ldots\wedge \overline x_n$, which can be seen as an element of $H^n_{1+z}(Q_{2n},\K_n^{MW})$. The result follows.
\end{proof}

\subsection{The concrete clutching construction}
We now apply the results of the abstract clutching construction with $Y = Q_{2n-1}$ and $X = Q_{2n} \cong \Sigma^1_s Y$.  Assume $k= \Z$ momentarily.  In that case, Theorem \ref{thm:abstractclutching} yields a surjective function
\[
cl_{GL_m}: \pi_0(\Sing^{\aone}GL_m(Q_{2n-1},1)) \longrightarrow \mathscr{V}_m(Q_{2n}).
\]
In particular, all rank $m$ vector bundles on $Q_{2n}$ are obtained from pointed morphisms from the quadric to $GL_m$.

We are particularly interested in the clutching function in the case where $m = n$.  The weak equivalence $Q_{2n} \cong {\pone}^{\sma n}$ of Theorem \ref{thm:main}, together with $\aone$-representability of algebraic K-theory \cite[\S 4 Theorem 3.13]{MV}, determines an isomorphism of reduced K-theory $\tilde{K}_0(Q_{2n}) \cong K_0(\Z) \cong \Z$. In fact if $k$ is (Noetherian) regular, then we conclude that $\tilde{K}_0(Q_{2n})$ is a free $K_0(k)$-module of rank $1$.\footnote{Note: these kinds of representability statements continue to hold if $k$ is not Noetherian.  As observed in \cite[Appendix C]{Hoyois}, algebraic K-theory is representable in the version of $\ho{k}$ we use.  Likewise, the proof of geometric representability (by the Grassmannian) goes through in this setting; one may modify \cite[Theorem 4]{SchlichtingTripathi} as necessitated by \cite[Remark 2]{SchlichtingTripathi} to check it when $k = \Z$ and then use base-change to treat the general case.}Our goal, accomplished in Theorem \ref{thm:generator}, is to describe an explicit {\em vector bundle} generator for this $K_0(k)$-module.

\begin{ex}
For small values of $n$, the generator of $\tilde{K}_0(Q_{2n})$ are given by classical geometric constructions.  For $n = 1$, the generator is simply the class of $\O(1)$, viewed as an element of $Pic(Q_2)$.  Alternatively, this line bundle is the line bundle associated with Hopf map $\eta: Q_3 \to Q_2$, which is a $\gm{}$-torsor.  Similarly, for $n = 2$, a generator is given by a Hopf bundle. Indeed, consider the associated vector bundle to the $SL_2$-torsor corresponding with the motivic Hopf map $\nu: Q_7 \to Q_4$.  The map $\nu$ can can be defined as follows: given a pair of $2 \times 2$-matrices $(M_1,M_2)$ such that $\det M_1 - \det M_2 = 1$, send $(M_1,M_2)$ to $(M_1M_2,\det M_2)$.  A corresponding generator can also be given for $Q_8$ in terms of the motivic Hopf map $\sigma$, which can be defined using split octonion algebras.
\end{ex}

\begin{rem}
For $n \geq 3$, it is easy to write down generators using the isomorphism $K_0(Q_{2n}) \cong G_0(Q_{2n})$ and taking the pushforward $i_*[\O_{Q_{2n}\setminus X_{2n}}]$, but an explicit projective resolution of this class yields only a virtual vector bundle representative.
\end{rem}

If ${\bf x} = (x_1,\ldots,x_n)$ and ${\bf y} = (y_1,\ldots,y_n)$ be elements of a ring $A$ such that ${\bf x}{\bf y}^t = 1$.  Attached to such a pair, Suslin inductively defined matrices $\alpha_n({\bf x},{\bf y}) \in GL_{2^{n-1}}$ \cite[\S 5]{SuslinStablyFree}.  The universal such example corresponds to a morphism $\alpha_n: Q_{2n-1} \to GL_{2^{n-1}}$.  By means of elementary row operations, the matrix $\alpha_n$ can be reduced to an invertible $n \times n$-matrix (non-uniquely).  Equivalently, the morphism $Q_{2n-1} \to GL_{2^{n-1}}$ lifts (non-uniquely) up to (naive) $\aone$-homotopy to a morphism $\beta_n: Q_{2n-1} \to GL_n$ \cite[p. p. 489 Point (b)]{SuslinStablyFree}.

\begin{rem}
A more topologically inclined reader may prefer the following description of the matrices $\alpha_n$.  Atiyah-Bott and Shapiro constructed an explicit generator of the topological K-theory group $\tilde{K}_0(S^n)$ in terms of Clifford algebras \cite{ABS}.  Suslin's matrix $\alpha$ can be thought of as an algebraic version of Clifford multiplication; this has been worked out in more detail in \cite{Chintala}.  The morphisms $\alpha_n$ (resp. $\beta_n$) are closely related to algebraic K-theory.  Indeed, for any integer $i$, there is an isomorphism of reduced K-theory $\tilde{K}_i(Q_{2n-1}) \cong \tilde{K}_{i-1}(\Z)$ using the fact that $Q_{2n-1} \cong \Sigma^{n-1}_s \gm{\sma n}$ and $\aone$-representability of algebraic $K$-theory.  Suslin showed \cite[Theorem 2.3]{SuslinMennicke} that this isomorphism is given by multiplication by $\alpha_n$ (or, equivalently, $\beta_n$); see \cite[\S 3]{AsokFaselKO} for more discussion of Suslin matrices in the context of $\aone$-homotopy theory.
\end{rem}

\begin{thm}
\label{thm:generator}
If $k$ is a (Noetherian) regular base ring, then for every $n \geq 1$, $cl_{GL_n}(\beta_n)$ determines a rank $n$ vector bundle on $Q_{2n}$; the stable isomorphism class of this vector bundle yields a generator of $\widetilde{K}_0(Q_{2n})$.
\end{thm}

\begin{proof}
The first statement follows from Theorem \ref{thm:abstractclutching}(i).  Indeed, $cl_{GL_n}(\beta_n)$ gives a vector bundle on $Q_{2n}$ over $\Spec \Z$, which we may then pullback to $Q_{2n}$ over $\Spec k$.

For the second statement, recall that under these hypotheses $\widetilde{K}_0(Q_{2n})$ is a free $K_0(k)$-module of rank $1$.  By functoriality, it suffices to produce a generator over $\Spec \Z$.  In that case, we appeal to functoriality of the clutching construction for the map $GL_n \to GL$, the fact that the clutching construction is given by suspension and appeal to Suslin's results.  Indeed, it suffices to show that $cl_{GL}(\beta_n)$ is a generator of $\widetilde{K}_0(Q_{2n})$.  However, Suslin showed that the map $\beta_n: Q_{2n-1} \to GL$ is a generator of $[Q_{2n-1},GL]_{\aone} \cong \Z$; this is contained in \cite[Theorem 2.3]{SuslinMennicke}, but see \cite[Theorem 3.4.1]{AsokFaselKO} for a translation in this language.  The statement then follows from the suspension isomorphism in K-theory together with the identification $[Q_{2n-1},GL]_{\aone} \cong [Q_{2n},BGL]_{\aone} \cong \widetilde{K}_0(Q_{2n})$.
\end{proof}

\begin{rem}
The results of \cite[Proposition 3.3.3]{AsokFaselKO} together with obstruction theory arguments can be used to show that the generators of $\widetilde{K}_0(Q_{2n})$ described in Theorem \ref{thm:generator} admit reductions of structure group in certain situations: if $n \equiv 0 \mod 4$, then the vector bundle $cl_{GL_n}(\beta_n)$ admits a reduction of structure group to the split orthogonal group, while if $n \equiv 2 \mod 4$, then it admits a reduction of structure group to the symplectic group.
\end{rem}

Now, using Theorem \ref{thm:explicitequivalence}, we show how to refine Theorem \ref{thm:abstractclutching} to obtain explicit cocycles representing vector bundles on $Q_{2n}$ from the associated clutching functions.  To this end, assume that $k$ is an (infinite) perfect field, and consider the commutative diagram
\[
\xymatrix{
V_{2n}^0 \ar[d]& V_{2n}^0 \cap V_{2n}^1 \ar[r]\ar[l] \ar[d] & V_{2n}^1\ar[d] \\
\ast & Q_{2n-1} \ar[l]\ar[r] & \ast.
}
\]
The induced map of homotopy pushouts gives a reinterpretation of the map $\varphi_n$ and thus is an $\aone$-weak equivalence by Theorem \ref{thm:explicitequivalence}.  Moreover, the induced $\aone$-weak equivalence $Q_{2n} \to \Sigma^1_s Q_{2n-1}$ factors through a morphism $Q_{2n} \to \Sigma^1_s (V_{2n}^0 \cap V_{2n}^1) \to \Sigma^1_s Q_{2n-1}$.  The following result, which can be viewed as a refinement of some of the results in \cite[\S 6]{Nori10}, is our concrete version of the clutching construction in the case of $Q_{2n}$.

\begin{thm}
\label{thm:concreteclutching}
Assume $k$ is a field, and $n,r \geq 2$ are integers and $G$ is an isotropic reductive $k$-group (in the sense of \textup{\cite[Definition 3.3.4]{AHWII}}).
\begin{enumerate}[noitemsep,topsep=1pt]
\item If $k$ is infinite and perfect, given a pointed $k$-morphism $f: Q_{2n-1} \to G$, the $G$-torsor $cl_{G}(f)$ on $Q_{2n}$ attached to $f$ is obtained, up to isomorphism, by gluing the trivial $G$-torsor on $V_{2n}^0$ with a trivial $G$-torsor on $V_{2n}^1$ along the function  $f \circ \psi_n$.
\item For arbitrary $k$, the rank $n$ vector bundle on $Q_{2n}$ corresponding to $cl_{GL_n}(\beta_n)$ is that associated with $\beta_n \circ \psi_n$ by the construction of the previous point.
\end{enumerate}
\end{thm}

\begin{proof}
Here the assumption that $k$ is infinite is used in conjunction with the assumption that $G$ is isotropic to appeal to Theorem \ref{thm:abstractclutching}(ii).  Considering the diagram above, point (1) follows from the functoriality assertion in Theorem \ref{thm:abstractclutching} combined with Theorem \ref{thm:explicitequivalence}.  In more detail, given a vector bundle on $[Q_{2n},BG]_{\aone}$, we can lift it to a pointed homotopy class $[(Q_{2n},0),(BG,\ast)]_{\aone}$.  Since the weak equivalence $Q_{2n} \to \Sigma^1_s Q_{2n-1}$ factors through $V_{2n}^0 \cap V_{2n}^1$, every element of $[(Q_{2n},0),(BG,\ast)]_{\aone}$ corresponds to an element of $[\Sigma^1_s Q_{2n-1},(BG,\ast)]_{\aone}$ and thus determines an element of $[\Sigma^1_s(V_{2n}^0 \cap V_{2n}^1),(BG,\ast)]_{\aone}$.  The adjoints to these elements are precisely a pointed map $f: Q_{2n-1} \to G$ and the composite map $f \circ \psi_n$.

The clutching construction proceeds by composing $\Sigma^1_s f$ with the canonical morphism $\Sigma^1_s G \to BG$.  The factorization through $V_{2n}^0 \cap V_{2n}^1$ in the previous paragraph, then corresponds, by contemplating the diagram before the statement, to gluing two homotopically constant maps $V_{2n}^i \to \ast \to BG$, i.e., trivial $G$-torsors on $V_{2n}^i$ along the morphism $f \circ \psi_n$ as claimed.

For Point (2), begin by observing that the function $cl_{GL_n}(\beta_n)$ over the field $k$ is obtained by base-change from a morphism of schemes over $\Spec \Z$.  In particular, the function $cl_{GL_n}(\beta_n)$ over $k$ is obtained by extending scalars from a function defined over the prime field $F \subset k$, which is perfect.  Now, since an algebraic closure $\bar{F}$ of the prime field is infinite and perfect, Point (2) follows by combining Theorem \ref{thm:generator} with the conclusion of Point (1).  However, by Galois descent, there is an equivalence between vector bundles over $F$ and $Gal(\bar{F}/F)$-invariant cocycles over $\bar{F}$.  Therefore, the vector bundle corresponding to $cl_{GL_n}(\beta_n)$ over $\bar{F}$ descends to a vector bundle over $F$.  Pulling back this bundle to $k$ yields the required vector bundle.
\end{proof}

\begin{rem}
\label{rem:descent}
Note that Point (2) makes sense even if the intersection $V_{2n}^0 \cap V_{2n}^1$ is empty.  More generally, using faithfully flat descent techniques, it is possible to build $G$-torsors over fields that arise from clutching functions given by $k$-morphisms $f: Q_{2n-1} \to G$ defined over perfect subfields $F \subset k$.  For example, if $k$ is a finite field, every $G$-torsor over $Q_{2n}$ under an isotropic reductive $G$ is described by means of a clutching function $Q_{2n-1} \to G$.  Lacking any applications of the result in this form at the moment, we leave the proof of this fact to the reader.
\end{rem}

\begin{ex}
\label{ex:Hopfbundle}
Unwinding the definitions, the Suslin matrix $Q_3 \to SL_2$ corresponds to the identity map $SL_2 \to SL_2$ under the isomorphism of $Q_3$ with $SL_2$ given by sending $(x_1,x_2,y_1,y_2)$ to the $2 \times 2$-matrix, $\begin{pmatrix} x_1 & x_2 \\ -y_2 & y_1 \end{pmatrix}$.  From this, one can see that the resulting vector bundle on $Q_4$ is precisely the Hopf bundle.
\end{ex}

We close with one final application of our results.  Recall that, since the inclusion $X_{2n} \hookrightarrow Q_{2n}$ has complement of codimension $n$, whenever $n \geq 2$ the restriction functor on categories of vector bundles is fully-faithful \cite[Lemma 2.4]{ADBundle}.  Combining this fact with Theorem \ref{thm:generator} one obtains the following result, which is a generalization of \cite[Corollary 3.1]{ADBundle}.

\begin{cor}
If $k$ is a (Noetherian) regular base ring, then for every $n \geq 2$, the $\aone$-contractible $k$-scheme $X_{2n}$ has a non-trivial vector bundle.
\end{cor}

\begin{footnotesize}
\bibliographystyle{alpha}
\bibliography{quadrics}

\begin{thebibliography}{AHW15b}

\bibitem[ABS64]{ABS}
M.~F. Atiyah, R.~Bott, and A.~Shapiro.
\newblock Clifford modules.
\newblock {\em Topology}, 3(suppl. 1):3--38, 1964.

\bibitem[AD07]{ADContractible}
A.~Asok and B.~Doran.
\newblock Unipotent groups and some {${\mathbb A}^1$}-contractible smooth
  schemes.
\newblock {\em Int. Math. Res. Pap.}, 5, 2007.
\newblock Art. ID rpm005.

\bibitem[AD08]{ADBundle}
A.~Asok and B.~Doran.
\newblock Vector bundles on contractible smooth schemes.
\newblock {\em Duke Math. J.}, 143(3):513--530, 2008.

\bibitem[AF14a]{AsokFaselSpheres}
A.~Asok and J.~Fasel.
\newblock Algebraic vector bundles on spheres.
\newblock {\em J. Topology}, 7(3):894--926, 2014.
\newblock doi:10.1112/jtopol/jtt046.

\bibitem[AF14b]{AsokFaselThreefolds}
A.~Asok and J.~Fasel.
\newblock A cohomological classification of vector bundles on smooth affine
  threefolds.
\newblock {\em Duke Math. J.}, 163(14):2561--2601, 2014.

\bibitem[AF14c]{AsokFaselKO}
A.~Asok and J.~Fasel.
\newblock An explicit {${\mathbf {KO}}$}-degree map and applications.
\newblock {\em Preprint}, available at \url{http://arxiv.org/abs/1403.4588},
  2014.

\bibitem[AF15]{AsokFaselmotiviccohomotopy}
A.~Asok and J.~Fasel.
\newblock Euler class groups are motivic cohomotopy groups.
\newblock {\em In preparation}, 2015.

\bibitem[AHW15a]{AHW}
A.~Asok, M.~Hoyois, and M.~Wendt.
\newblock Affine representability results in {${\mathbb A}^1$}-homotopy theory
  {I}: vector bundles.
\newblock {\em Preprint}, available at \url{http://arxiv.org/abs/1506.07093},
  2015.

\bibitem[AHW15b]{AHWII}
A.~Asok, M.~Hoyois, and M.~Wendt.
\newblock Affine representability results in {${\mathbb A}^1$}-homotopy theory
  {II}: principal bundles and homogeneous spaces.
\newblock {\em Preprint}, available at \url{http://arxiv.org/abs/1507.08020},
  2015.

\bibitem[Bor91]{Borel}
A.~Borel.
\newblock {\em Linear algebraic groups}, volume 126 of {\em Graduate Texts in
  Mathematics}.
\newblock Springer-Verlag, New York, 1991.

\bibitem[Chi15]{Chintala}
V.~Chintala.
\newblock On {S}uslin matrices and their connection to spin groups.
\newblock {\em Doc. Math.}, 20:531--550, 2015.

\bibitem[CTHK97]{CTHK}
J.-L. Colliot-Th{\'e}l{\`e}ne, R.~T. Hoobler, and B.~Kahn.
\newblock The {B}loch-{O}gus-{G}abber theorem.
\newblock In {\em Algebraic {$K$}-theory ({T}oronto, {ON}, 1996)}, volume~16 of
  {\em Fields Inst. Commun.}, pages 31--94. Amer. Math. Soc., Providence, RI,
  1997.

\bibitem[Cur71]{Curtis}
E.~B. Curtis.
\newblock Simplicial homotopy theory.
\newblock {\em Advances in Math.}, 6:107--209 (1971), 1971.

\bibitem[DI08]{DuggerIsaksenHopf}
D.~Dugger and D.~Isaksen.
\newblock {\em Unpublished work}, 2008.

\bibitem[Fas11]{Fasel08b}
J.~Fasel.
\newblock Some remarks on orbit sets of unimodular rows.
\newblock {\em Comment. Math. Helv.}, 86(1):13--39, 2011.

\bibitem[Fas15]{FaselMurthy}
J.~Fasel.
\newblock On the number of generators of ideals in polynomial rings.
\newblock {\em Preprint}, available at \url{http://arxiv.org/abs/1507.05734},
  2015.

\bibitem[Gro61]{EGAIII1}
A.~Grothendieck.
\newblock \'{E}l\'ements de g\'eom\'etrie alg\'ebrique. {III}. \'{E}tude
  cohomologique des faisceaux coh\'erents. {I}.
\newblock {\em Inst. Hautes \'Etudes Sci. Publ. Math.}, 11:167, 1961.

\bibitem[Gup13]{Gupta2}
N.~Gupta.
\newblock On {Z}ariski's cancellation problem in positive characteristic.
\newblock {\em Adv. Math.}, 264:296--307, 2013.

\bibitem[Gup14]{Gupta}
N.~Gupta.
\newblock On the cancellation problem for the affine space {$\Bbb{A}^3$} in
  characteristic {$p$}.
\newblock {\em Invent. Math.}, 195(1):279--288, 2014.

\bibitem[Har66]{HartshorneRD}
R.~Hartshorne.
\newblock {\em Residues and duality}.
\newblock Lecture notes of a seminar on the work of A. Grothendieck, given at
  Harvard 1963/64. With an appendix by P. Deligne. Lecture Notes in
  Mathematics, No. 20. Springer-Verlag, Berlin, 1966.

\bibitem[Hov99]{Hovey}
M.~Hovey.
\newblock {\em Model {C}ategories}, volume~63 of {\em Math. {S}urveys and
  {M}onographs}.
\newblock American Mathematical Society, Providence, RI, 1999.

\bibitem[Hoy14]{Hoyois}
M.~Hoyois.
\newblock A quadratic refinement of the {G}rothendieck-{L}efschetz-{V}erdier
  trace formula.
\newblock {\em Algebr. Geom. Topol.}, 14(6):3603--3658, 2014.

\bibitem[Lur11]{DAGXI}
J.~Lurie.
\newblock Derived algebraic geometry {XI}: Descent theorems.
\newblock 2011.
\newblock {\em Preprint,} available at
  \url{http://www.math.harvard.edu/~lurie/papers/DAG-XI.pdf}.

\bibitem[Mor05]{MStable}
F.~Morel.
\newblock The stable {${\mathbb A}^1$}-connectivity theorems.
\newblock {\em $K$-Theory}, 35(1-2):1--68, 2005.

\bibitem[Mor08]{Morelquadrics}
F.~Morel.
\newblock {\em Personal communication} dated {April 3}, 2008.

\bibitem[Mor12]{MField}
F.~Morel.
\newblock {\em {${\mathbb A}^1$}-algebraic topology over a field}, volume 2052
  of {\em Lecture Notes in Mathematics}.
\newblock Springer, Heidelberg, 2012.

\bibitem[MV99]{MV}
F.~Morel and V.~Voevodsky.
\newblock {${\mathbb A}^1$}-homotopy theory of schemes.
\newblock {\em Inst. Hautes \'Etudes Sci. Publ. Math.}, 90:45--143 (2001),
  1999.

\bibitem[MVW06]{MVW}
C.~Mazza, V.~Voevodsky, and C.~Weibel.
\newblock {\em Lecture notes on motivic cohomology}, volume~2 of {\em Clay
  Mathematics Monographs}.
\newblock American Mathematical Society, Providence, RI, 2006.

\bibitem[NRS10]{Nori10}
M.~V. Nori, R.~A. Rao, and R.~G. Swan.
\newblock Non-self-dual stably free modules.
\newblock In J.-L. Colliot-Th{\'e}l{\`e}ne, S.~Garibaldi, R.~Sujatha, and
  V.~Suresh, editors, {\em Quadratic forms, linear algebraic groups, and
  cohomology}, volume~18 of {\em Dev. Math.}, pages 315--324. Springer, New
  York, 2010.

\bibitem[PW10]{PaninWalterPontryaginClasses}
I.~Panin and C.~Walter.
\newblock Quaternionic grassmannians and pontryagin classes in algebraic
  geometry.
\newblock {\em Preprint} available at \url{http://arxiv.org/abs/1011.0649},
  2010.

\bibitem[Sch15]{Schlichting}
M.~Schlichting.
\newblock Euler class groups, and the homology of elementary and special linear
  groups.
\newblock {\em Preprint}, available at \url{http://arxiv.org/abs/1502.05424},
  2015.

\bibitem[ST15]{SchlichtingTripathi}
M.~Schlichting and G.~S. Tripathi.
\newblock Geometric models for higher {G}rothendieck-{W}itt groups in
  {$\Bbb{A}^1$}-homotopy theory.
\newblock {\em Math. Ann.}, 362(3-4):1143--1167, 2015.

\bibitem[Sus77]{SuslinStablyFree}
A.~A. Suslin.
\newblock Stably free modules.
\newblock {\em Mat. Sb. (N.S.)}, 102(144)(4):537--550, 632, 1977.

\bibitem[Sus82]{SuslinMennicke}
A.~A. Suslin.
\newblock Mennicke symbols and their applications in the {$K$}-theory of
  fields.
\newblock In {\em Algebraic {$K$}-theory, {P}art {I} ({O}berwolfach, 1980)},
  volume 966 of {\em Lecture Notes in Math.}, pages 334--356. Springer,
  Berlin-New York, 1982.

\bibitem[Voe03a]{VMod2}
V.~Voevodsky.
\newblock Motivic cohomology with {${\bf Z}/2$}-coefficients.
\newblock {\em Publ. Math. Inst. Hautes \'Etudes Sci.}, 98:59--104, 2003.

\bibitem[Voe03b]{VRed}
V.~Voevodsky.
\newblock Reduced power operations in motivic cohomology.
\newblock {\em Publ. Math. Inst. Hautes \'Etudes Sci.}, 98:1--57, 2003.

\bibitem[Voe10]{VCancellation}
V.~Voevodsky.
\newblock Cancellation theorem.
\newblock {\em Doc. Math.}, Extra volume: Andrei A. Suslin sixtieth
  birthday:671--685, 2010.

\bibitem[Woo93]{WoodQuad}
R.M.W. Wood.
\newblock Polynomial maps of affine quadrics.
\newblock {\em Bull. London Math. Soc.}, 25(5):491--497, 1993.

\end{thebibliography}
\end{footnotesize}
\Addresses
\end{document}